\documentclass{amsart}
\usepackage{algorithm,algorithmicx,algpseudocode}
\usepackage{amsfonts,amsmath,amssymb,amsthm}
\usepackage[title]{appendix}
\usepackage{booktabs}
\usepackage[dvipdfmx]{color}
\usepackage[dvipdfmx]{graphicx}
\usepackage{hyperref}
\usepackage{listings}
\usepackage{manyfoot}
\usepackage{mathrsfs,mathtools}
\usepackage{multirow}
\usepackage{textcomp}
\usepackage{url}
\usepackage[dvipdfmx]{xcolor}
\usepackage[all]{xy}

\theoremstyle{definition}
\newtheorem{Def}{Definition}[section]
\newtheorem{Rmk}[Def]{Remark}
\newtheorem{Exp}[Def]{Expectation}
\theoremstyle{plain}
\newtheorem{Lem}[Def]{Lemma}
\newtheorem{Prop}[Def]{Proposition}
\newtheorem{Thm}[Def]{Theorem}
\newtheorem{Cor}[Def]{Corollary}

\makeatletter

\let\c@algorithm\c@Def\let\cl@algorithm\cl@Def
\def\fld{\hspace{-0.1mm}\mathbb{F}_{p^2}\hspace{-0.4mm}}
\makeatother
\title[Listing superspecial genus-3 curves using Richelot isogeny graphs]{\vspace*{-10mm}Listing superspecial curves of genus three using Richelot isogeny graphs}
\author[R.\,Ohashi,\ \ H.\,Onuki,\ \ M.\,Kudo,\ \ R.\,Yoshizumi\ \ \and \ K.\,Nuida]{Ryo Ohashi,\ \ Hiroshi Onuki,\ \ Momonari Kudo,\\Ryo Yoshizumi,\ \ \and \ Koji Nuida}

\begin{document}
\begin{abstract}
In algebraic geometry, superspecial curves are important research objects. 
While the number of superspecial genus-3 curves in characteristic $p$ is known, the number of hyperelliptic ones among them has not been determined even for small $p$. 
In this paper, in order to compute the latter number, we give an explicit algorithm for computing the Richelot isogeny graph of superspecial principally polarized abelian varieties of dimension 3 using theta functions.
In particular, one can determine whether a given vertex in the graph corresponds to the Jacobian of a genus-3 curve or not, and restore the defining equation of such a genus-3 curve from its theta constants.
Our algorithm enables efficient enumeration of superspecial genus-3 curves, as all operations can be performed in $\fld$.
By implementing the algorithm in Magma, we successfully counted the number of hyperelliptic curves among them for all primes $11 \leq p < 100$.
\end{abstract}
\maketitle\vspace{-6mm}

{\small \textit{Key words:\ \,superspecial curve, abelian threefold, Richelot isogeny, theta function}\par
\textit{2020 Mathematics Subject Classification(s):\ \,14G17, 14H45, 14K02, 14K25}}

\section{Introduction}
Throughout this paper, by a \emph{curve} we mean a non-singular projective variety of dimension 1 and \emph{isomorphisms} between two curves are defined over an algebraically closed field.
A curve $C$ over an algebraically closed field $k$ of characteristic $p > 0$ is called \emph{superspecial} if the Jacobian of $C$ is isomorphic to a product of supersingular elliptic curves.
In recent years, superspecial curves have been used for isogeny-based cryptography (e.g., \cite{G2SIDH}, \cite{CDS}, \cite{FESTA}, etc), and studies on such curves are important.

For a pair $(g,p)$ of integers, it is known that the number of isomorphism classes of superspecial curves of genus $g$ in characteristic $p$ is finite, and determining that number is a fundamental problem.
This problem is already solved for genus $g \leq 3$; Deuring~\cite{Deuring} found the number of all supersingular elliptic curves (i.e., $g=1$), and the number for $g=2$ is {given} by Ibukiyama-Katsura-Oort~\cite{IKO}.
In the case $g=3$, Oort~\cite{Oort} and Ibukiyama~\cite{Ibukiyama} showed that there exists a superspecial genus-3 curve in arbitrary characteristic $p \geq 3$.
An explicit formula for the number of such curves is found in \cite[Theorem 3.10(d)]{Brock}.
As {a related problem}, one may ask how many of superspecial curves of genus-$g$ in characteristic $p$ are \emph{hyperelliptic}.
This problem is open even for genus $g=3$.
In particular, it is still unknown whether there exists a superspecial hyperelliptic curve of genus 3 in arbitrary characteristic $p \geq 7$.

A naive approach for enumerating superspecial hyperelliptic curves of genus 3 is as follows: any genus-3 hyperelliptic curve in characteristic $p \geq 3$ can be written in the form $y^2 = x(x-1)\prod_{i=1}^5(x-\lambda_i)$ and such a curve is superspecial if and only if its Cartier-Manin matrix is zero.
By treating $\lambda_1,\ldots,\lambda_5$ as variables and computing all solutions to the system of equations defined by the condition that each entry of the Cartier-Manin matrix vanishes, one can obtain all the superspecial hyperelliptic curves of genus 3.
Due to the high computational cost of computing Gröbner bases, this method is not very efficient. 

Therefore, we instead utilize the connectedness of the \emph{Richelot isogeny graph} (see Section \ref{graph} for definitions) of superspecial principally polarized abelian varieties of dimension $g$, denoted by $\mathcal{G}_g(2,p)$.
While the graph $\mathcal{G}_2(2,p)$ can be computed using known formulae (cf. \cite[$\S 3$]{CDS}), this is no longer the case when $g \geq 3$ (partial formulae such as those in \cite{Smith2} exist for $g=3$, but they are not sufficient to obtain the entire graph).
In this paper, we will describe an explicit algorithm for computing $\mathcal{G}_3(2,p)$ using \emph{theta functions}.
As outlined in \cite[Appendix F]{SQIsignHD}, theta functions allow us to compute Richelot isogenies between abelian varieties (the case for $g=2$ has indeed been implemented by \cite{DMPR}).
One of our contributions is to explicitly formulate and implement this algorithm in the case of $g=3$.

Our overall strategy for listing superspecial genus-3 curves is as follows: we start from one principally polarized superspecial abelian threefold, and compute vertices connected to it via edges in $\mathcal{G}_3(2,p)$, repeating this process to generate all vertices.
The important point is that the computation of Richelot isogenies, the verification of whether the codomain is isomorphic to the Jacobian of a genus-3 curve, and the reconstruction of such a curve can all be carried out using \emph{theta constants}.
We note that this process is derived based on discussions over $\mathbb{C}$, but it also works correctly in characteristic $p > 7$.
Consequently, we obtain the following main theorem:
\begin{Thm}\label{Main1}
There exists an algorithm (Algorithm \ref{main}) for listing superspecial genus-3 curves in characteristic $p > 7$ with $\widetilde{O}(p^6)\hspace{-0.3mm}$ {arithmetic} operations over $\fld$.
\end{Thm}
\noindent Our algorithm enables us to produce explicit defining equations of all superspecial curves of genus 3 efficiently, compared to the aforementioned Gröbner bases method (this complexity would be exponential with respect to $p$ in worst case).

Executing our algorithm, we succeeded in proving the existence of superspecial hyperelliptic curves of genus 3 for small $p$ as follows:
\begin{Thm}\label{Main2}
The number of isomorphism classes of superspecial genus-3 hyperelliptic curves for $11 \leq p < 100$ is summarized in Table \ref{enumerate} of Section \ref{computation}.
Moreover, there exists such a curve in any characteristic $p$ with $7 \leq p < 10000$.
\end{Thm}
\noindent We note that the upper bounds on $p$ in Theorem \ref{Main2} can be increased easily.
From our computational results, we propose a conjecture exists a superspecial hyperelliptic curve of genus 3 in arbitrary characteristic $p \geq 7$.

The rest of this paper is organized as follows: \,Section \ref{preliminaries} is dedicated to reviewing some facts on analytic theta functions.
In Section \ref{graph}, we recall the definition and properties of superspecial Richelot isogeny graph $\mathcal{G}_g(2,p)$.
We then explain how to compute the graph $\mathcal{G}_3(2,p)$ in Section \ref{dim3}, and finally Algorithm \ref{main} gives its explicit construction.
In Section \ref{computation}, we state computational results about the existence and the number of superspecial hyperelliptic genus-3 curves in small $p$.\medskip

\noindent\emph{Acknowledgements.}
\ The authors are grateful to Damien Robert and Everett Howe for their helpful comments and suggestions on an earlier version of this paper.
The authors would also like to thank the anonymous reviewer for their constructive and valuable feedback, which significantly improved the quality of this paper.
This work was supported by JSPS Grant-in-Aid for Young Scientists 25K17225, 24K20770 and 23K12949.
This research was supported by the WISE program (MEXT) at Kyushu University.\smallskip

\noindent\emph{Competing interests.}
\ The authors have no conflicts of interest to declare.\smallskip

\noindent\emph{Data availability.}
\ No data was used for the research described in the article.

\section{Preliminaries}\label{preliminaries}
In this section, we summarize background knowledge necessary for later sections.
Abelian varieties and isogenies between them before Section \ref{g=3} are defined over $\mathbb{C}$, but in Section \ref{graph} we consider those over a field of odd characteristic.

\subsection{Abelian varieties and their isogenies}
In this subsection, we shall review the theory of complex abelian varieties and isogenies.
A \emph{complex abelian variety} of dimension $g$ is isomorphic to $\mathbb{C}^g/(\mathbb{Z}^g + \varOmega\mathbb{Z}^g)$, where $\varOmega$ is an element of the Siegel upper half-space\vspace{-0.5mm}
\[
    \mathcal{H}_g \coloneqq \{\varOmega \in {\rm Mat}_g(\mathbb{C})\,\mid{\,}^t\!\varOmega = \varOmega,\,{\rm Im}\,\varOmega > 0\}.\medskip
\]
Such $\varOmega \in \mathcal{H}_g$ is called a (small)\,\emph{period matrix} of the abelian variety.
First, we recall a characterization of isomorphisms between abelian varieties:
\begin{Def}
For a commutative ring $R$ and an integer $g \geq 1$, we define a group\smallskip
\[
    {\rm Sp}_{2g}(R) \coloneqq \{M \in {\rm GL}_{2g}(R)\,\mid\,ME{\,}^t\!M = E\}, \quad E \coloneqq
    \begin{pmatrix}
        0 & {\mathbf 1}_g\\
        -{\mathbf 1}_g & 0
    \end{pmatrix}.\smallskip
\]
We say that a matrix $M \in {\rm GL}_{2g}(R)$ is \emph{symplectic} over $R$ if $M \in {\rm Sp}_{2g}(R)$.
\end{Def}\smallskip

For a matrix $M = \begin{psmallmatrix} \alpha & \beta\\[-0.3mm] \gamma & \delta\end{psmallmatrix} \in {\rm Sp}_{2g}(\mathbb{Z})$, the action on the Siegel upper half-space\smallskip
\begin{align*}
    \begin{split}
    {\rm Sp}_{2g}(\mathbb{Z}) \times \mathcal{H}_g  &\longrightarrow \mathcal{H}_g\\[-1mm]
    (M,\varOmega) &\longmapsto (\alpha\varOmega+\beta)(\gamma\varOmega+\delta)^{-1} \eqqcolon M.\varOmega\\[0.3mm]
    \end{split}
\end{align*}
is well-defined. Moreover, the map
\[
    \mathbb{C}^g \longrightarrow \mathbb{C}^g\,;\,z \longmapsto {}^t(\gamma\varOmega+\delta)^{-1}z \eqqcolon M.z\medskip
\]
induces an isomorphism $\mathbb{C}^g/(\mathbb{Z}^g + \varOmega\mathbb{Z}^g) \rightarrow \mathbb{C}^g/(\mathbb{Z}^g + M.\varOmega\mathbb{Z}^g)$.
\begin{Prop}[{\cite[Proposition 8.1.3]{Lange}}]\label{isom}
The principally polarized abelian varieties of dimension $g$ with period matrices $\varOmega,\varOmega' \in \mathcal{H}_g$ are isomorphic to each other if and only if there exists a symplectic matrix $M \in {\rm Sp}_{2g}(\mathbb{Z})$ such that $\varOmega' = M.\varOmega$.
\end{Prop}\medskip

An \emph{isogeny} between abelian varieties is a surjective homomorphism with a finite kernel.
If there exists an isogeny $A \rightarrow B$, then the dimensions of $A$ and $B$ are equal to each other, and we say that $A$ and $B$ are \emph{isogenous}.
Two isogenies $A \rightarrow B$ with the same kernel are equivalent up to an automorphism of $B$, thus we identify them.
\begin{Prop}\label{maxisonum}
For an abelian variety $A$ of dimension $g$ and a prime integer $\ell$, the number of maximal isotropic subgroups $G$ of $A[\ell]$ is equal to
\[
    N_g(\ell) \coloneqq \prod_{k=1}^g(\ell^k+1).\smallskip
\]
\end{Prop}
\begin{proof}
See \cite[Lemma 2]{CD}, for example.
\end{proof}
\noindent An isogeny whose kernel is such a group $G$ the above is called an \emph{$(\ell,\dots,\ell)$-isogeny}.
In particular, one can see that the $(\ell,\ldots,\ell)$-isogeny\smallskip
\begin{align*}
    \begin{split}
    \mathbb{C}^g/(\mathbb{Z}^g + \varOmega\mathbb{Z}^g) &\longrightarrow \mathbb{C}^g/(\mathbb{Z}^g + \ell\varOmega\mathbb{Z}^g)\\[-0.5mm]
    z &\longmapsto \ell z\\[-0.3mm]
    \end{split}
\end{align*}
has the kernel $\frac{1}{\ell}\mathbb{Z}^g/\mathbb{Z}^g$.

\subsection{Analytic theta functions}\label{theta}
The \emph{({analytic}) theta function} associated to characteristics $a,b \in \mathbb{Q}^g$ is defined by
\[
    \theta\bigl[\begin{smallmatrix}a\\b\end{smallmatrix}\bigr](z,\varOmega) \coloneqq \sum_{n \in \mathbb{Z}^g} \exp\bigl(\pi i{\,}^t\!(n+a)\varOmega(n+a)+ 2\pi i{\,}^t\!(n+a)(z+b)\bigr)\smallskip
\]
for $z \in \mathbb{C}^g$ and $\varOmega \in \mathcal{H}_g$.
In the following, we restrict our attention to theta functions associated to characteristics $a,b \in \{0,1/2\}^g$.
Then, we note that the identity
\[
    \theta\bigl[\begin{smallmatrix}a+\alpha\\b+\beta\end{smallmatrix}\bigr](z,\varOmega) = (-1)^{4{}^t\!a\beta}\theta\bigl[\begin{smallmatrix}a\\b\end{smallmatrix}\bigr](z,\varOmega)\smallskip
\]
holds for any $\alpha, \beta \in \mathbb{Z}^g$.
\begin{Lem}\label{even-odd}
The theta function $\theta\bigl[\begin{smallmatrix}a\\b\end{smallmatrix}\bigr](z,\varOmega)$ with characteristics $a,b \in \{0,1/2\}^g$ is an even function of $z$ if and only if $4{}^t\!ab \in \mathbb{Z}$ is even; otherwise, it is odd.
\end{Lem}
\begin{proof}
The first assertion follows from\vspace{-0.5mm}
\[
    \theta\bigl[\begin{smallmatrix}a\\b\end{smallmatrix}\bigr](-z,\varOmega) = (-1)^{4{}^t\!ab}\theta\bigl[\begin{smallmatrix}a\\b\end{smallmatrix}\bigr](z,\varOmega).\smallskip
\]
The second assertion follows from that $4{}^t\!ab$ is an integer for all $a,b \in \{0,1/2\}^g$.
\end{proof}\smallskip

The following two formulae are very fundamental, and generate a lot of relations among theta functions:
\begin{Thm}[Riemann's theta formula]\label{Riemann}
Let $m_k \coloneqq \begin{psmallmatrix}a_k\\b_k\end{psmallmatrix}$ with $a_k,b_k \in \{0,1/2\}^g$, and we define $n_k$ via
\[
    \begin{pmatrix}
    n_1\\
    n_2\\
    n_3\\
    n_4
    \end{pmatrix} \coloneqq \frac{1}{2}
    \begin{pmatrix}
        {\mathbf 1}_{2g} & \hspace{2.7mm}{\mathbf 1}_{2g} & \hspace{2.7mm}{\mathbf 1}_{2g} & \hspace{2.7mm}{\mathbf 1}_{2g}\\
        {\mathbf 1}_{2g} & \hspace{2.7mm}{\mathbf 1}_{2g} & -{\mathbf 1}_{2g} & -{\mathbf 1}_{2g}\\
        {\mathbf 1}_{2g} & -{\mathbf 1}_{2g} & \hspace{2.7mm}{\mathbf 1}_{2g} & -{\mathbf 1}_{2g}\\
        {\mathbf 1}_{2g} & -{\mathbf 1}_{2g} & -{\mathbf 1}_{2g} & \hspace{2.7mm}{\mathbf 1}_{2g}
    \end{pmatrix}\!
    \begin{pmatrix}
    m_1\\
    m_2\\
    m_3\\
    m_4
    \end{pmatrix}.
\]
Then, we have the equation
\[
    \prod_{k=1}^4\theta[m_k](0,\varOmega) =
    \frac{1}{2^g}\!\sum_{\alpha,\beta \in \{0,1/2\}^g}\!(-1)^{4{}^t\!a\beta}\prod_{k=1}^4\theta[n_k+\begin{psmallmatrix}\alpha\\\beta\end{psmallmatrix}](0,\varOmega).
\]
\end{Thm}
\begin{proof}
Apply \cite[Chapter\,IV, Theorem 1]{Igusa} for $z_1 = z_2 = z_3 = z_4 = 0$.
\end{proof}
\begin{Thm}[Duplication formula]\label{duplication}
For any $a,b \in \{0,1/2\}^g$, we have
\[
    \theta\bigl[\begin{smallmatrix}a\\b\end{smallmatrix}\bigr](z,\varOmega)^2 = \frac{1}{2^g}\sum_{\beta \in \{0,1/2\}^g}\!(-1)^{4{}^t\!a\beta}\theta\bigl[\begin{smallmatrix}0\\\beta\end{smallmatrix}\bigr](0,\varOmega/2)\theta\bigl[\begin{smallmatrix}0\\b+\beta\end{smallmatrix}\bigr](z,\varOmega/2).
\]
\end{Thm}
\begin{proof}
Apply \cite[Chapter\,IV, Theorem 2]{Igusa} for $m_1 = m_2$ and $z_1 = z_2 = z$.
\end{proof}\smallskip

In particular, if a matrix $\varOmega \in \mathcal{H}_g$ is written as
\[
    \varOmega = \begin{pmatrix}
    \varOmega_1 & 0\\
    0 & \varOmega_2
    \end{pmatrix},\ {\rm where}\ \,\varOmega_1 \in \mathcal{H}_{g_1} \text{ and }\, \varOmega_2 \in \mathcal{H}_{g_2}
\]
with $g_1 + g_2 = g$, then we have the following lemma:
\begin{Lem}[{\cite[Lemma F.3.1]{SQIsignHD}}]\label{diagonal}
With the above notations, we write $z \coloneqq (z_1,z_2)$ where $z_1 \in \mathbb{C}^{g_1},\,z_2 \in \mathbb{C}^{g_2}$. Then, we have the equation\vspace{0.5mm}
\[
    \theta\bigl[\begin{smallmatrix}a_1 & a_2\\b_1 & b_2\end{smallmatrix}\bigr](z,\varOmega) = \theta\bigl[\begin{smallmatrix}a_1\\b_1\end{smallmatrix}\bigr](z_1,\varOmega_1)\theta\bigl[\begin{smallmatrix}a_2\\b_2 \end{smallmatrix}\bigr](z_2,\varOmega_2)\medskip
\]
for all characteristics $a_1,b_1 \in \{0,1/2\}^{g_1}\!$ and $a_2,b_2 \in \{0,1/2\}^{g_2}$.
\end{Lem}

Let us return to cases with a general (not necessarily diagonal) matrix $\varOmega \in \mathcal{H}_g$.
For a symplectic matrix $M = \begin{psmallmatrix} \alpha & \beta\\[-0.3mm] \gamma & \delta\end{psmallmatrix} \in {\rm Sp}_{2g}(\mathbb{Z})$ and characteristics $a,b \in \{0,1/2\}^g$, we define two vectors
\begin{equation}\label{index}
    M.a \coloneqq a\hspace{0.3mm}^t\hspace{-0.3mm}\delta - b\,^t\hspace{-0.3mm}\gamma + \frac{1}{2}(\gamma\hspace{0.3mm}^t\hspace{-0.3mm}\delta)_0, \quad M.b \coloneqq b\,^t\hspace{-0.5mm}\alpha - a\hspace{0.3mm}^t\hspace{-0.3mm}\beta + \frac{1}{2}(\alpha\hspace{0.3mm}^t\hspace{-0.3mm}\beta)_0
\end{equation}
and a scalar\vspace{-0.5mm}
\begin{equation}\label{scalar}
    k(M,a,b) \coloneqq (a\hspace{0.3mm}^t\hspace{-0.3mm}\delta - b\,^t\hspace{-0.3mm}\gamma) \,\cdot\hspace{0.3mm}^t\hspace{-0.3mm}(b\,^t\hspace{-0.3mm}\alpha - a\hspace{0.3mm}^t\hspace{-0.3mm}\beta + (\alpha\hspace{0.3mm}^t\hspace{-0.3mm}\beta)_0) - a\hspace{0.3mm}^tb\smallskip
\end{equation}
where $(\alpha\hspace{0.3mm}^t\hspace{-0.4mm}\beta)_0,\hspace{0.3mm}(\gamma\hspace{0.5mm}^t\hspace{-0.3mm}\delta)_0$ denote the row vector of the diagonal elements of $(\alpha\hspace{0.3mm}^t\hspace{-0.3mm}\beta),\hspace{0.3mm}(\gamma\hspace{0.5mm}^t\hspace{-0.3mm}\delta)$.
Then, the matrix $M$ acts on theta functions in the following way:\vspace{-0.5mm}

\begin{Thm}[Theta transformation formula]\label{trans}
For any $a,b \in \{0,1/2\}^g$, we have
\[
    \theta\bigl[\begin{smallmatrix}M.a\\\hspace{-0.3mm}M.b\end{smallmatrix}\bigr](M.z,M.\varOmega) = c \cdot e^{\pi ik(M,a,b)}\theta\bigl[\begin{smallmatrix}a\\b\end{smallmatrix}\bigr](z,\varOmega),\smallskip
\]
where $c \in \mathbb{C}^\times\!$ is a constant which does not depend $a$ or $b$.
\end{Thm}
\begin{proof}
This directly follows from \cite[Theorem 8.6.1]{Lange}.
\end{proof}

Let $n \geq 1$ be an integer and define
\begin{align*}
    {\varGamma_0(n)} &\coloneqq {\bigl\{M = \begin{psmallmatrix} \alpha & \beta\\[-0.3mm] \gamma & \delta\end{psmallmatrix} \in {\rm Sp}_{2g}(\mathbb{Z})\,\mid\,\gamma \equiv \mathbf{0}_g \!\!\pmod{n}\bigr\}},\\
    \varGamma_n &\coloneqq \bigl\{M = \begin{psmallmatrix} \alpha & \beta\\[-0.3mm] \gamma & \delta\end{psmallmatrix}\in {\rm Sp}_{2g}(\mathbb{Z})\,\mid\,\alpha \equiv \delta \equiv {\mathbf 1}_g,\ \beta \equiv \gamma \equiv {\mathbf 0}_g\!\pmod{n}\bigr\},\\
    \varGamma_{n,2n} &\coloneqq \bigl\{M = \begin{psmallmatrix} \alpha & \beta\\[-0.3mm] \gamma & \delta\end{psmallmatrix} \in \varGamma_n\,\mid\,{\rm diag}(\hspace{0.3mm}^t\hspace{-0.3mm}\alpha\gamma) \equiv {\rm diag}(\hspace{0.3mm}^t\hspace{-0.3mm}\beta\hspace{0.3mm}\delta) \equiv 0\!\pmod{2n}\},
\end{align*}
which are all subgroups of ${\rm Sp}_{2g}(\mathbb{Z})$.
\begin{Cor}\label{Gamma}
For any $a,b \in \{0,1/2\}^g$, we have the equations
\begin{alignat*}{3}
    \frac{\theta\bigl[\begin{smallmatrix}a\\b\end{smallmatrix}\bigr](0,M.\varOmega)^4}{\theta\bigl[\begin{smallmatrix}0\\0\end{smallmatrix}\bigr](0,M.\varOmega)^4} &= \frac{\theta\bigl[\begin{smallmatrix}a\\b\end{smallmatrix}\bigr](0,\varOmega)^4}{\theta\bigl[\begin{smallmatrix}0\\0\end{smallmatrix}\bigr](0,\varOmega)^4}& \quad & \text{if}\ M \in \varGamma_2,\\
    \frac{\theta\bigl[\begin{smallmatrix}a\\b\end{smallmatrix}\bigr](0,M.\varOmega)^2}{\theta\bigl[\begin{smallmatrix}0\\0\end{smallmatrix}\bigr](0,M.\varOmega)^2} &= \frac{\theta\bigl[\begin{smallmatrix}a\\b\end{smallmatrix}\bigr](0,\varOmega)^2}{\theta\bigl[\begin{smallmatrix}0\\0\end{smallmatrix}\bigr](0,\varOmega)^2}& \quad & \text{if}\ M \in \varGamma_{2,4}, \text{ and}\\
    \frac{\theta\bigl[\begin{smallmatrix}a\\b\end{smallmatrix}\bigr](0,M.\varOmega)}{\theta\bigl[\begin{smallmatrix}0\\0\end{smallmatrix}\bigr](0,M.\varOmega)} &= \frac{\theta\bigl[\begin{smallmatrix}a\\b\end{smallmatrix}\bigr](0,\varOmega)}{\theta\bigl[\begin{smallmatrix}0\\0\end{smallmatrix}\bigr](0,\varOmega)}& \quad & \text{if}\ M \in \varGamma_{4,8}.\\[-5mm]
\end{alignat*}
\end{Cor}
\begin{proof}
It follows from Theorem \ref{trans} that
\[
    \frac{\theta\bigl[\begin{smallmatrix}a\\b\end{smallmatrix}\bigr](0,M.\varOmega)}{\theta\bigl[\begin{smallmatrix}0\\0\end{smallmatrix}\bigr](0,M.\varOmega)} = \frac{\theta\bigl[\begin{smallmatrix}a\\b\end{smallmatrix}\bigr](0,\varOmega)}{\theta\bigl[\begin{smallmatrix}0\\0\end{smallmatrix}\bigr](0,\varOmega)}\exp\bigl(\pi i a\hspace{0.3mm}^t\hspace{-0.3mm}\beta\hspace{0.3mm}\delta\,^t\hspace{-0.3mm}a-\pi i b\hspace{0.3mm}^t\hspace{-0.3mm}\gamma\hspace{0.3mm}\alpha\hspace{0.3mm}^tb-2\pi ia(\hspace{0.3mm}^t\hspace{-0.3mm}\alpha-\mathbf{1}_g)\hspace{0.3mm}^tb\bigr).
\]
Hence, we obtain this assertion by simple computations.
\end{proof}

\subsection{Theta null-points of abelian threefolds}\label{g=3}
In this subsection, we focus only on the case that $A$ is a principally polarized abelian threefold, and we discuss theta functions associated to $A$.
Thanks to \cite[Corollary 11.8.2]{Lange}, such $A$ is classified into one of the following four types:
\begin{equation*}
    \begin{split}
    \left\{
    \begin{array}{l}
    {\rm Type\ 1\!:\ }\,A \cong {\rm Jac}(C)\ \text{where}\ C\ \text{is\ a\ smooth\ plane\ quartic},\\[0.5mm]
    {\rm Type\ 2\!:\ }\,A \cong {\rm Jac}(C)\ \text{where}\ C\ \text{is\ a\ hyperelliptic\ curve\ of\ genus\ } 3,\\[0.5mm]
    {\rm Type\ 3\!:\ }\,A \cong E \times {\rm Jac}(C)\ \text{where}\ E\ \text{and}\ C\text{\ is\ a genus-1 and genus-2 curve},\\[0.5mm]
    {\rm Type\ 4\!:\ }\,A \cong E_1 \times E_2 \times E_3\ \text{where}\ E_1,E_2 \text{\ and\ } E_3 \text{\ are\ elliptic\ curves}.
    \end{array}
    \right.
    \end{split}
\end{equation*}
First of all, we introduce the following notations for the sake of simplicity.
\begin{Def}\label{notation}
For a period matrix $\varOmega \in \mathcal{H}_3$ of $A$, we write\smallskip
\[
    \vartheta_i(z,\varOmega) \coloneqq \theta\bigl[\begin{smallmatrix}a_1 & a_2 & a_3\\b_1 & b_2 & b_3\end{smallmatrix}\bigr](z,\varOmega), \quad i \coloneqq 2b_1 + 4b_2 + 8b_3 + 16a_1 + 32a_2 + 64a_3,\medskip
\]
where $a_1,a_2,a_3,b_1,b_2$, and $b_3$ belong to $\{0,1/2\}$.
Note that $i$ runs over $\{0,\ldots,63\}$.
For each $i$, the value $\vartheta_i(0,\varOmega)$ is written as $\vartheta_i$, and called the $i$-th \emph{theta constant}.
\end{Def}

We consider a principally polarized abelian threefold $A \cong \mathbb{C}^3/(\mathbb{Z}^3+\varOmega\mathbb{Z}^3)$, then we call $[\vartheta_0:\vartheta_1:\cdots:\vartheta_{63}] \in \mathbb{P}^{63}(\mathbb{C})$ a \emph{level-4 (analytic) theta null-point} of $A$. Also, we call $[\vartheta_0^2:\vartheta_1^2:\cdots:\vartheta_{63}^2] \in \mathbb{P}^{63}(\mathbb{C})$ a \emph{squared theta null-point} of $A$.
Note that even for isomorphic principally polarized abelian threefolds, the level-4 theta null-points may differ depending on the choice of $\varOmega$ modulo the action of $\varGamma_{4,8}$.
\begin{Def}
We define the sets\smallskip
\[
    I_{\rm even} \coloneqq \left\{
    \begin{array}{l}
        0,\,1,\,2,\,3,\,4,\,5,\,6,\,7,\,8,\,10,\,12,\,14,\,16,\,17,\,20,\,21,\,24,\,27,\,28,\\
        31,\,32,\,33,\,34,\,35,\,40,\,42,\,45,\,47,\,48,\,49,\,54,\,55,\,56,\,59,\,61,\,62
    \end{array}
    \right\}
\]
and $I_{\rm odd} \coloneqq \{0,\ldots,63\}\!\smallsetminus\!I_{\rm even}$.
Here, we remark that $\#I_{\rm even} = 36$ and $\#I_{\rm odd} = 28$.
\end{Def}\smallskip

It follows from Lemma \ref{even-odd} that $\vartheta_i(z)$ is even (resp. odd) if and only if $i \in I_{\rm even}$ (resp. $i \in I_{\rm odd}$).
Therefore, we have that $\vartheta_i = 0$ for all $i \in I_{\rm odd}$.
\begin{Lem}\label{invariant}
Let $S_{\rm van}$ be the set of vanishing indices $i \in I_{\rm even}$, that is,
\[
    S_{\rm van} \coloneqq \{i \in I_{\rm even} \mid \vartheta_i = 0\}.\smallskip
\]
Then, the cardinality $N_{\rm van}$ of $S_{\rm van}$ is invariant for isomorphic principally polarized abelian threefolds.
\end{Lem}
\begin{proof}
This follows from Theorem \ref{trans} together with Proposition \ref{isom}.
\end{proof}

Then, we can identify by $N_{\rm van}$ which of the above 4 types $A$ corresponds to (note that this does not depend on a period matrix $\varOmega$ of $A$ by the above lemma).
\begin{Prop}\label{van_num}
The possible values of $N_{\rm van}$ are $0,1,6,9$ and moreover\vspace{0.5mm}
\[
    \left\{
        \begin{array}{l}
        N_{\rm van} =0\ \Leftrightarrow\,A \cong {\rm Jac}(C)\ \text{where}\ C\ \text{is\ a\ smooth\ plane\ quartic},\\[0.5mm]
        N_{\rm van} =1\ \Leftrightarrow\,A \cong {\rm Jac}(C)\ \text{where}\ C\ \text{is\ a\ hyperelliptic\ curve\ of\ genus\ } 3,\\[0.5mm]
        N_{\rm van} =6\ \Leftrightarrow\,A \cong E \times {\rm Jac}(C)\ \text{where}\ E\ \text{and}\ C\text{\ is\ a genus-1 and genus-2 curve},\\[0.5mm]
        N_{\rm van} =9\ \Leftrightarrow\,A \cong E_1 \times E_2 \times E_3\ \text{where}\ E_1,E_2 \text{\ and\ } E_3 \text{\ are\ elliptic\ curves}.
        \end{array}
    \right.
\]
\end{Prop}
\begin{proof}
The first assertion follows from \cite[Theorem 3.1]{Glass}\ (the cases where $N_{\rm van} > 9$ correspond to singular abelian threefolds).
Also, for the second assertion, the cases where $N_{\rm van}=0,1$ follows from \cite[Theorem 2]{Oyono} or \cite[Remark 6.3]{Pieper}.
In the following, we show the remaining cases (i.e., $N_{\rm van} = 6,9$).\par
Let $A$ be a principally polarized abelian threefold isomorphic to the (principally polarized) product of an elliptic curve $A_1$ and an abelian surface $A_2$.
Denote $\varOmega_1,\varOmega_2$, and $\varOmega$ to be the period matrices of $A_1,A_2$, and $A$ respectively. 
Theorem \ref{isom} shows that there exists $M \in {\rm Sp}_6(\mathbb{Z})$ such that $M.\varOmega = {\rm diag}(\varOmega_1,\varOmega_2)$, hence we may assume that $\varOmega = {\rm diag}(\varOmega_1,\varOmega_2)$ using Lemma \ref{invariant}.
Then, it follows from Lemma \ref{diagonal} that\smallskip
\[
    \theta\bigl[\begin{smallmatrix}a_1 & a_2 & a_3\\b_1 & b_2 & b_3\end{smallmatrix}\bigr](z,\varOmega) = \theta\bigl[\begin{smallmatrix}a_1\\b_1\end{smallmatrix}\bigr](z_1,\varOmega_1)\theta\bigl[\begin{smallmatrix}a_2 & a_3\\b_2 & b_3\end{smallmatrix}\bigr](z_2,\varOmega_2), \quad z \coloneqq (z_1,z_2).\medskip
\]
As known in \cite[Remark 3.5]{LR} or \cite[$\S 16.4$]{Robert}, a simple (resp. not simple) principally polarized abelian surface has no (resp. exactly one) vanishing even theta constant.
This shows that $N_{\rm van} = 6$ if $A_2$ is simple and $N_{\rm van} = 9$ otherwise, as desired.
\end{proof}

In the rest of this subsection, suppose that $A$ corresponds to a genus-3 curve $C$.
Let us review how to restore a defining equation of $C$ from theta constants of $A$.
\begin{Prop}\label{hyp_restore}
If $S_{\rm van} = \{61\}$ holds, then $A$ is isomorphic to the Jacobian of a genus-3 hyperelliptic curve $C: y^2 = x(x-1)\prod_{i=1}^5(x-\lambda_i)$ where
\begin{equation}\label{Rosenhain}
    \lambda_1 \coloneqq \frac{\vartheta_{42}^2\vartheta_2^2}{\vartheta_5^2\vartheta_{45}^2},\ \,\lambda_2 \coloneqq \frac{\vartheta_{42}^2\vartheta_3^2}{\vartheta_4^2\vartheta_{45}^2},\ \,\lambda_3 \coloneqq \frac{\vartheta_{27}^2\vartheta_{42}^2}{\vartheta_{45}^2\vartheta_{28}^2},\ \,\lambda_4 \coloneqq \frac{\vartheta_2^2\vartheta_{49}^2}{\vartheta_{54}^2\vartheta_5^2},\ \,\lambda_5 \coloneqq \frac{\vartheta_3^2\vartheta_0^2}{\vartheta_7^2\vartheta_4^2}.
\end{equation}
\end{Prop}
\begin{proof}
Apply \cite[Theorem 1.1]{Takase} for $g=3$, or see \cite[Theorem 4]{Shaska}.
\end{proof}

On the other hand, even in the case that $C$ is a plane quartic, we can also restore a defining equation of $C$ as follows: taking a period matrix $\varOmega \in \mathcal{H}_3$ of $A = {\rm Jac}(C)$, it follows from Proposition \ref{van_num} that $\vartheta_i \neq 0$ holds for all $i \in\hspace{-0.3mm} I_{\rm even}$.
Then, we define the following non-zero values
\begin{alignat}{3}\label{aij}
    a_{11} &\coloneqq\frac{\vartheta_{12}\vartheta_5}{\vartheta_{40}\vartheta_{33}}, & \quad a_{21} &\coloneqq \frac{\vartheta_{27}\vartheta_5}{\vartheta_{40}\vartheta_{54}}, & \quad a_{31} &\coloneqq -\frac{\vartheta_{27}\vartheta_{12}}{\vartheta_{33}\vartheta_{54}}, \nonumber\\
    a_{12} &\coloneqq \frac{\vartheta_{21}\vartheta_{28}}{\vartheta_{49}\vartheta_{56}}, & \quad a_{22} &\coloneqq \frac{\vartheta_2\vartheta_{28}}{\vartheta_{49}\vartheta_{47}}, & \quad a_{32} &\coloneqq \frac{\vartheta_2\vartheta_{21}}{\vartheta_{56}\vartheta_{47}},\\
    a_{13} &\coloneqq \frac{\vartheta_7\vartheta_{14}}{\vartheta_{35}\vartheta_{42}}, & \quad
    a_{23} &\coloneqq \frac{\vartheta_{16}\vartheta_{14}}{\vartheta_{35}\vartheta_{61}}, & \quad a_{33} &\coloneqq \frac{\vartheta_{16}\vartheta_7}{\vartheta_{42}\vartheta_{61}}\nonumber\\[-4.5mm]\nonumber
\end{alignat}
computed as in \cite[p.\,15]{Fiorentino}.
Then, there exist $k_1,k_2,k_3,\lambda_1,\lambda_2,\lambda_3 \in \mathbb{C}^\times$ such that
\begin{equation}\label{matrix}
    \begin{pmatrix}
    \lambda_1a_{11} & \lambda_2a_{21} & \lambda_3a_{31}\\
    \lambda_1a_{12} & \lambda_2a_{22} & \lambda_3a_{32}\\
    \lambda_1a_{13} & \lambda_2a_{23} & \lambda_3a_{33}
    \end{pmatrix}\!\!
    \begin{pmatrix}
    k_1\\
    k_2\\
    k_3
    \end{pmatrix} \!=\! 
    \begin{pmatrix}
    1/a_{11} & 1/a_{21} & 1/a_{31}\\
    1/a_{12} & 1/a_{22} & 1/a_{32}\\
    1/a_{13} & 1/a_{23} & 1/a_{33}
    \end{pmatrix}\!\!
    \begin{pmatrix}
    \lambda_1\\
    \lambda_2\\
    \lambda_3
    \end{pmatrix} \!=\!
    \begin{pmatrix}
    -1\\
    -1\\
    -1
    \end{pmatrix}.
\end{equation}
\begin{Thm}[Weber's formula]\label{Weber}
In the above situation, there exist homogeneous linear polynomials $\xi_1,\xi_2,\xi_3 \in \mathbb{C}[x,y,z]$ such that
\[
    \left\{
    \begin{array}{l}
        \xi_1+\xi_2+\xi_3+x+y+z = 0,\\
        \!\displaystyle\frac{\xi_1}{a_{i1}} + \frac{\xi_2}{a_{i2}} + \frac{\xi_3}{a_{i3}} + k_i(a_{i1}x + a_{i2}y + a_{i3}z) = 0, \quad {\rm for\ all}\ i \in \{1,2,3\}.
    \end{array}
    \right.
\]
Then $A$ is isomorphic to the Jacobian of a plane quartic
\[
    C: (\xi_1x + \xi_2y - \xi_3z)^2 - 4\xi_1\xi_2xy = 0.
\]
\end{Thm}
\begin{proof}
See \cite[$\S 15$]{Weber}, \cite[Theorem 3.1]{NR}, or \cite[Proposition 2]{Fiorentino}.
\end{proof}

\subsection{Superspecial Richelot isogeny graph}\label{graph}
Unlike the previous subsections, let us consider abelian varieties over an algebraically closed field of characteristic $p \geq 3$ in this subsection.
Let us review the definition of a superspecial curve:
\begin{Def}
An abelian variety $A$ is called \emph{superspecial} if $A$ is isomorphic to the product of supersingular elliptic curves (without a polarization).
A superspecial curve is a curve whose Jacobian is a superspecial abelian variety.
\end{Def}

We note that if two curves $C,C'\hspace{-0.3mm}$ are not isomorphic to each other, then neither are their Jacobians (as principally polarized abelian varieties), by Torelli's theorem.
Let $\Lambda_{g,p}$ be the list of isomorphism classes (over $\overline{\mathbb{F}}_p$) of superspecial genus-$g$ curves.
The numbers $\hspace{-0.1mm}\#\Lambda_{g,p}\hspace{-0.1mm}$ for $g \leq 3$ are given by \cite{Deuring}, \cite{IKO}, and \cite{Hashimoto}, respectively (see \cite{Kudo} for a survey).
However, it is not known how many of those curves are hyperelliptic for genus $g=3$.
In the following, we explain how to generate the list $\Lambda_{g,p}$, by using the isogeny graph of superspecial abelian varieties:
\begin{Def}
Let $\ell$ be a prime integer.
The \emph{superspecial isogeny graph} $\mathcal{G}_g(\ell,p)$ is the directed multigraph with the following properties:
\begin{itemize}
\item Its vertices are isomorphism classes of all principally polarized superspecial abelian varieties of dimension $g$ over $\overline{\mathbb{F}}_p$.
\item Its edges are isomorphism classes of all $(\ell,\ldots,\ell)$-isogenies over $\overline{\mathbb{F}}_p$ between two vertices (isogenies $\phi: A \rightarrow B$ and $\psi: A' \rightarrow B'$ are isomorphic if {there exists an isomorphism $\beta: B \rightarrow B'$ such that $\psi = \beta \circ \phi$}).
\end{itemize}
Then, the graph $\mathcal{G}_g(\ell,p)$ is an $N_g(\ell)$-regular multigraph by Proposition \ref{maxisonum}. 
\end{Def}

The fact that the Pizer's graph $\mathcal{G}_1(\ell,p)$ is connected is well-known (cf. \cite{Pizer}), but more surprisingly, the same assertion holds for $\mathcal{G}_g(\ell,p)$ as a  generalization:\vspace{-0.5mm}
\begin{Thm}\label{connected}
The graph $\mathcal{G}_g(\ell,p)$ is connected for each $\ell \neq p$.\vspace{-0.5mm}
\end{Thm}
\begin{proof}
See \cite[Theorem 34]{JZ} or \cite[Corollary 2.8]{ATY}.
\end{proof}\smallskip

In the following, {we fix $\ell = 2$, in which case the graph} $\mathcal{G}_g(2,p)$ is often called the \emph{superspecial Richelot isogeny graph} of dimension $g$.
The list $\Lambda_{2,p}$ can be generated by using $\mathcal{G}_2(2,p)$ as in the following pseudocode (cf. \cite[Algorithm 5.1]{KHH}).\vspace{-1.1mm}

\begin{algorithm}[htbp]
\begin{algorithmic}[1]
\caption{{\sf ListSspGenus2Curves}}\label{ListUpSSpGenus2Curves}
\renewcommand{\algorithmicrequire}{\textbf{Input:}}
\renewcommand{\algorithmicensure}{\textbf{Output:}}
\Require A prime integer $p > 5$ and the list $\Lambda_{1,p}$ of isomorphism classes of supersingular elliptic curves
\Ensure The list $\Lambda_{2,p}$ of isomorphism classes of superspecial genus-2 curves
\State Let $\mathcal{L}$ be the list of isomorphism classes of $E_1 \times E_2$ with $E_1,E_2 \in \Lambda_{1,p}$
\State Set $k \leftarrow 1$
\Repeat
\State Let $A$ be the $k$-th abelian surface of $\mathcal{L}$
\ForAll{abelian surfaces $A'$ which are $(2,2)$-isogenous to $A$}
\If{$A'$ is not isomorphic to any abelian surface of $\mathcal{L}$}
\State Add $A'$ to the end of $\mathcal{L}$
\EndIf
\EndFor
\State Set $k \leftarrow k+1$
\Until{$k > \#\mathcal{L}$}\\
\Return the list of genus-2 curves whose Jacobian belongs to $\mathcal{L}$
\end{algorithmic}
\end{algorithm}\vspace{-0.7mm}

\noindent In Algorithm \ref{ListUpSSpGenus2Curves}, we can compute the codomains of Richelot isogenies of line 4 by using formulae in \cite[Proposition 4]{HLP} and \cite[Chapter 8]{Smith}.
In addition, isomorphism tests in line 5 can be done using \emph{Igusa invariants} of a genus-2 curve (cf. \cite{Igusa-inv}).\vspace{-0.5mm}
\begin{Rmk}
The structure of the graph $\mathcal{G}_2(2,p)$ was studied by many authors ({e.g., \cite{CDS},\,\cite{KT}, and \cite{FS2}}).
In particular, it is known that its subgraph of Jacobians is also connected from the {work of Florit-Smith}~\cite[Theorem 7.2]{FS1}.
Their result may enable the construction of a more efficient algorithm for listing superspecial genus-2 curves {than} Algorithm \ref{ListUpSSpGenus2Curves}.
\end{Rmk}

Similarly to the case of $g=2$, we want to list superspecial curves of genus $g \geq 3$.
However, no explicit formula is known for the computations of $(2,\ldots,2)$-isogenies in general, unlike the genus-2 case.
In the next section, we will propose an algorithm for computing $(2,2,2)$-isogenies using theta functions.

\section{Computing the Richelot isogeny graph of dimension 3}\label{dim3}
The goal of this section is to provide an explicit algorithm for listing superspecial genus-3 curves.
For this purpose, we describe the following subroutines in the first three subsections:\smallskip
\begin{itemize}
\item[---] Section \ref{one}: Given a squared theta null-point of an abelian threefold $A$, we give an algorithm (Algorithm \ref{compute_isogeny}) for computing a squared theta null-point of \emph{one} abelian threefold $A'$ that is $(2,2,2)$-isogenous to $A$.\smallskip
\item[---] Section \ref{all}: Given a squared theta null-point of an abelian threefold $A$, we give an algorithm (Algorithm \ref{compute_isogenies}) for computing squared theta null-points of \emph{all} the abelian threefolds $A'$ that are $(2,2,2)$-isogenous to $A$.\smallskip
\item[---] Section \ref{restore}: Given a squared theta null-point of {$A$} which corresponds to a genus-3 curve, we give an algorithm (Algorithms \ref{restore_hyper} and \ref{restore_quartic}) to produce {a} defining equation for this curve.\smallskip
\end{itemize}
These computations use the theta constants introduced in Section \ref{g=3}, and remain valid over a field of odd characteristic (the details will be discussed in Section \ref{algebraic}).
In Section \ref{complete}, we present our main algorithm (Algorithm \ref{main}) and finally provide a proof of Theorem \ref{Main1}.

\subsection{Following an edge from a given vertex}\label{one}
In this subsection, we consider a complex abelian threefold $A$ with the period matrix $\varOmega \in \mathcal{H}_3$.
{First, we provide the following fact about a level-4 theta null-point of $A$:}
\begin{Lem}\label{01234567}
The equation
\[
    \vartheta_0\vartheta_1\vartheta_2\vartheta_3 - \vartheta_4\vartheta_5\vartheta_6\vartheta_7 = \vartheta_{32}\vartheta_{33}\vartheta_{34}\vartheta_{35}\smallskip
\]
holds with the same notation {as} in Definition \ref{notation}.
\end{Lem}
\begin{proof}
{By applying Riemann's theta formula (Theorem \ref{Riemann}) to the four matrices}
\begin{alignat*}{3}
    m_1 &\coloneqq \begin{pmatrix}
    0 & 0 & 0\\
    0 & 0 & 0
    \end{pmatrix}, \quad & m_2 &\coloneqq \begin{pmatrix}
    0 & 0 & 0\\
    1/2 & 0 & 0
    \end{pmatrix},\\
    m_3 &\coloneqq \begin{pmatrix}
    0 & 0 & 0\\
    0 & 1/2 & 0
    \end{pmatrix}, \quad & m_4 &\coloneqq \begin{pmatrix}
    0 & 0 & 0\\
    1/2 & 1/2 & 0
    \end{pmatrix}\\[-4.7mm]
\end{alignat*}
the claim follows after tedious computations.
\end{proof}\medskip

Next, suppose that a squared theta null-point $[\vartheta_0^2:\vartheta_1^2:\cdots:\vartheta_{63}^2] \in \mathbb{P}^{63}(\mathbb{C})$ of $A$ is given.
Squaring both sides of the equation in Lemma \ref{01234567}, we have\smallskip
\begin{equation}\label{product}
    2\vartheta_0\vartheta_1\vartheta_2\vartheta_3\vartheta_4\vartheta_5\vartheta_6\vartheta_7 = \vartheta_0^2\vartheta_1^2\vartheta_2^2\vartheta_3^2 + \vartheta_4^2\vartheta_5^2\vartheta_6^2\vartheta_7^2 - \vartheta_{32}^2\vartheta_{33}^2\vartheta_{34}^2\vartheta_{35}^2,\medskip
\end{equation}
which means that $\prod_{j=0}^7 \vartheta_j$ can be determined by the given squared theta null-point.
On the other hand, we can take any square root of $\vartheta_j^2$ for all $j \in \{0,\ldots,7\}$ so that equation (\ref{product}) holds, by the following proposition:
\begin{Prop}\label{construct}
Suppose that $\vartheta_0 \neq 0$.
For an integer $k \in \{1,\ldots,6\}$, there exists a matrix $M_k \in {\varGamma_0(2)}$ such that
\[
    \frac{\vartheta_j(0,M_k.\varOmega)}{\vartheta_0(0,M_k.\varOmega)} = \left\{
    \begin{array}{ll}
        -\vartheta_j/\vartheta_0 \ & {\rm if}\ j = k\ {\rm or}\ 7,\\
        \vartheta_j/\vartheta_0 \ & {\rm otherwise}
    \end{array}
    \right.\smallskip
\]
for all integers $j \in \{1,\ldots,7\}$.
\end{Prop}
\begin{proof}
It follows from Corollary \ref{Gamma} that all symplectic matrix in $\varGamma_{4,8}$ (resp. $\varGamma_{2,4}$) leave the values (resp. squares) of $\vartheta_j/\vartheta_0$ invariant for all $j$.
Hence, let us construct a matrix $M_k \in \varGamma_{2,4}\!\smallsetminus\!\varGamma_{4,8}$ that changes the signs of only $\vartheta_k/\vartheta_0$ and $\vartheta_7/\vartheta_0$.
For each integer $k \in \{1,\ldots,6\}$, we define the $6$-tuple\smallskip
\[
    (c_{11},c_{12},c_{13},c_{22},c_{23},c_{33}) \coloneqq \left\{
    \begin{array}{l}
        (1,1,1,0,0,0) \quad \text{if }\,k = 1,\\
        (0,1,0,1,1,0) \quad \text{if }\,k = 2,\\
        (0,1,0,0,0,0) \quad \text{if }\,k = 3,\\
        (0,0,1,0,1,1) \quad \text{if }\,k = 4,\\
        (0,0,1,0,0,0) \quad \text{if }\,k = 5,\\
        (0,0,0,0,1,0) \quad \text{if }\,k = 6
    \end{array}\right.
\]
and $M_k \coloneqq \begin{psmallmatrix} \alpha & \beta\\[-0.3mm] \gamma & \delta\end{psmallmatrix}$ where\vspace{-1.5mm}
\begin{equation}\label{abcd}
    \alpha \coloneqq \mathbf{1}_3, \quad \beta \coloneqq \mathbf{0}_3, \quad \gamma \coloneqq \begin{pmatrix}
    4c_{11} & 2c_{12} & 2c_{13}\\
    2c_{12} & 4c_{22} & 2c_{23}\\
    2c_{13} & 2c_{23} & 4c_{33}
    \end{pmatrix}, \quad \delta \coloneqq \mathbf{1}_3.\smallskip
\end{equation}
It is obvious that $\alpha\hspace{0.3mm}^t\!\beta = {\mathbf 0}_3,\,\gamma\hspace{0.3mm}^t\hspace{-0.3mm}\delta = \gamma$ are both symmetric and $\alpha\hspace{0.3mm}^t\hspace{-0.3mm}\delta - \beta\hspace{0.3mm}^t\hspace{-0.3mm}\gamma = {\mathbf 1}_3$ holds, therefore $M_k$ belongs to ${\rm Sp}_6(\mathbb{Z})$ from a well-known criterion (cf. \cite[Lemma 8.2.1]{Lange}).
One can check that $M_k$ is the desired matrix since for any $b \in \{0,1/2\}^3$, we have
\[
    \frac{\theta\bigl[\begin{smallmatrix}0\\b\end{smallmatrix}\bigr](0,M_k.\varOmega)}{\theta\bigl[\begin{smallmatrix}0\\0\end{smallmatrix}\bigr](0,M_k.\varOmega)} = \frac{\theta\bigl[\begin{smallmatrix}0\\b\end{smallmatrix}\bigr](0,\varOmega)}{\theta\bigl[\begin{smallmatrix}0\\0\end{smallmatrix}\bigr](0,\varOmega)}\exp(-\pi i b\gamma\hspace{0.3mm}^tb)\smallskip
\]
by Theorem \ref{trans}.
{Moreover, it follows from $\gamma \equiv \mathbf{0}_3 \pmod{2}$ that $M_k \in \varGamma_0(2)$.}
\end{proof}\smallskip

Summarizing the discussions, one can compute $[\vartheta_0:\vartheta_1:\cdots:\vartheta_7] \in \mathbb{P}^7(\mathbb{C})$ from a given squared theta null-point of $A$.
Note that if there exists an index $j \in \{0,\ldots,7\}$ such that $\vartheta_j = 0$, then one can take any square root of $\vartheta_k^2$ for all $k \in \{0,\ldots,7\}\!\smallsetminus\!\{j\}$ (see Remark \ref{0vanish} below a proof when $\vartheta_0=0$).
Consider the $(2,2,2)$-isogeny\smallskip
\begin{align*}
    \begin{split}
    A = \mathbb{C}^3\hspace{-0.2mm}/(\mathbb{Z}^3 + \varOmega\mathbb{Z}^3) &\longrightarrow \mathbb{C}^3\hspace{-0.2mm}/(\mathbb{Z}^3 + 2\varOmega\mathbb{Z}^3) = B\\[-0.5mm]
    z &\longmapsto 2z
    \end{split}
\end{align*}
whose kernel is $\frac{1}{2}\mathbb{Z}^3/\mathbb{Z}^3$, and we explain how to compute a squared theta null-point of the codomain $B$.
Applying Theorem \ref{duplication}
for $z=0$ and $\varOmega \mapsto 2\varOmega$, we obtain\smallskip
\begin{equation}\label{dupl_g=3}
    \theta\bigl[\begin{smallmatrix}a\\b\end{smallmatrix}\bigr](0,2\varOmega)^2 = \frac{1}{2^3}\sum_{\beta \in \{0,1/2\}^3}\!(-1)^{4\hspace{0.3mm}^t\hspace{-0.3mm}a\beta}\theta\bigl[\begin{smallmatrix}0\\\beta\end{smallmatrix}\bigr](0,\varOmega)\theta\bigl[\begin{smallmatrix}0\\b+\beta\end{smallmatrix}\bigr](0,\varOmega)\smallskip
\end{equation}
for all characteristics $a,b \in \{0,1/2\}^3$.
Since the right-hand side of (\ref{dupl_g=3}) is obtained by the values $\vartheta_0,\ldots,\vartheta_7$, we obtain a squared theta null-point of $B$.

\begin{Rmk}\label{0vanish}
{If $\vartheta_0 = \vartheta_1 = \cdots = \vartheta_7 = 0$, then by (\ref{dupl_g=3}) we obtain $\theta\bigl[\begin{smallmatrix}a\\b\end{smallmatrix}\bigr](0,2\varOmega)^2 = 0$ for all $a,b \in \{0,1/2\}^3$, which contradicts Proposition \ref{van_num}.
Hence, even in the case where $\vartheta_0=0$, one can choose an index $i \in \{1,\ldots,7\}$ such that $\vartheta_i \neq 0$.
In addition, for each $k \in \{1,\ldots,7\}\!\smallsetminus\!\{i\}$, one can construct a symplectic matrix $M_{i,k} \in \varGamma_{2,4}\!\smallsetminus\!\varGamma_{4,8}$ such that
\[
    \frac{\vartheta_j(0,M_{i,k}.\varOmega)}{\vartheta_i(0,M_{i,k}.\varOmega)} = \left\{
    \begin{array}{ll}
        -\vartheta_j/\vartheta_i \ & {\rm if}\ j = 0\ {\rm or}\ k,\\
        \vartheta_j/\vartheta_i \ & {\rm otherwise}
    \end{array}
    \right.\medskip
\]
for all integers $j \in \{0,\ldots,7\}\!\smallsetminus\!\{i\}$, in a similar way to the proof of Proposition \ref{construct}.
This means that, even in the case where $\vartheta_0=0$, we can take any square root of $\vartheta_k^2$ for all $k \in \{1,\ldots,7\}$.}
\end{Rmk}

As a summary of the statements in this subsection, we give an explicit algorithm (Algorithm \ref{compute_isogeny} below) for computing a squared theta null-point of the codomain of a $(2,2,2)$-isogeny from the given abelian threefold $A$:\vspace{-0.7mm}
\begin{algorithm}[htbp]
\begin{algorithmic}[1]
\caption{{\sf \ ComputeOneRichelotIsogeny}}\label{compute_isogeny}
\renewcommand{\algorithmicrequire}{\textbf{Input:}}
\renewcommand{\algorithmicensure}{\textbf{Output:}}
\Require A squared theta null-point $[\vartheta_i^2]_i \in \mathbb{P}^{63}$ of a p.p.a.v.$\,A$ of dimension 3
\Ensure A squared theta null-point of $B$, one of the p.p.a.v.'s which are $(2,2,2)$-isogenous to $A$
\If{$\vartheta_0^2,\ldots,\vartheta_7^2$ are all non-zero}
\State Choose a square root of $(\vartheta_k/\vartheta_0)^2$ for all $k \in \{1,\ldots,6\}$
\State Compute $\vartheta_7/\vartheta_0$ from equation (\ref{product}) 
\Else
\State Find an index $j \in \{0,\ldots,7\}$ such that $\vartheta_j^2 \neq 0$
\State Choose a square root of $(\vartheta_k/\vartheta_j)^2$ for all $k \in \{0,\ldots,7\}\!\smallsetminus\!\{j\}$
\EndIf
\For{$i \in \{0,\ldots,63\}$}
\State Write $i = {a_3a_2a_1b_3b_2b_1}_{(2)}$ in binary and initialize ${\vartheta'}_{\!\!i}^2 \!\leftarrow 0$
\For{$j \in \{0,\ldots,7\}$}
\State Write $j = {\beta_3\beta_2\beta_1}_{(2)}$ {in binary} and $k \leftarrow i+j \pmod{8}$
\State Let ${\vartheta'}_{\!\!i}^2 \!\leftarrow {\vartheta'}_{\!\!i}^2 + (-1)^{a_1\beta_1+a_2\beta_2+a_3\beta_3} \cdot \vartheta_j\vartheta_k$
\EndFor
\EndFor\\
\Return $[{\vartheta'}_{\!\!0}^2:{\vartheta'}_{\!\!1}^2:\cdots:{\vartheta'}_{\!\!63}^2] \in \mathbb{P}^{63}$
\end{algorithmic}
\end{algorithm}\vspace{-0.7mm}

\noindent We note that it suffices to perform the loop in lines 8--14 only for $i \in I_{\rm even}$ instead, since ${\vartheta'}_{\!\!i}^2 \!= 0$ holds for all $i \in I_{\rm odd}$.

\subsection{Following all edges from a given vertex}\label{all}
We continue to assume that a squared theta null-point of an abelian threefold $A$ with the period matrix $\varOmega \in \mathcal{H}_3$ is given, and let $P,Q,R \in A$ as follows:\vspace{-0.5mm}
\[
    \xymatrix@R=-1pt{
        A \ar[r]^{\hspace{-10mm}\cong} & \mathbb{C}^3/(\mathbb{Z}^3 + \varOmega\mathbb{Z}^3)\\
        P,Q,R \ar@{|->}[r] & (1/2,0,0),\,(0,1/2,0),\,(0,0,1/2)
    }\vspace{0.2mm}
\]
Then, the $(2,2,2)$-isogeny computed in Algorithm \ref{compute_isogeny} has the kernel $G = \langle P,Q,R \rangle$, which is determined by a given squared theta null-point of $A$ (arbitrary sign choices in line 2 or 6 do not affect the isomorphism class of the codomains;\, see Remark \ref{kernel} for details).
To compute distinct $(2,2,2)$-isogenies from $A$, it is necessary to act an appropriate matrix $M \in {\rm Sp}_6(\mathbb{Z})$ on a squared theta null-point of $A$.
Specifically, if we write $M = \begin{psmallmatrix} \alpha & \beta\\[-0.3mm] \gamma & \delta\end{psmallmatrix}$ and let $M.P,M.Q,M.R \in A$ corresponding to
\[
    (\alpha_{1j}/2,\alpha_{2j}/2,\alpha_{3j}/2) + \varOmega(\gamma_{1j}/2,\gamma_{2j}/2,\gamma_{3j}/2), \quad {\rm where}\ \alpha = (\alpha_{ij}),\ \gamma = (\gamma_{ij})\medskip
\]
for $j \in \{1,2,3\}$ respectively, then we can compute the $(2,2,2)$-isogeny whose kernel is given as $M.G \coloneqq \langle M.P,M.Q,M.R \rangle$ after acting $M$ on a squared theta null-point.
For any symplectic matrix $M \in {\rm Sp}_6(\mathbb{Z})$, it is well-known (cf. \cite[$\S 2.3.2$]{Cosset}) that $G$ is invariant (i.e., $M.G = G$) under the action by $M$ if and only if $M \in \varGamma_0(2)$.
Hence, we need to act by an element of ${\rm Sp}_6(\mathbb{Z})/\varGamma_0(2)$ on the squared theta null-point of $A$ to obtain a different $(2,2,2)$-isogeny from $A$.
Since there exist exactly $N_3(2) = 135$ different $(2,2,2)$-isogenies from $A$ (cf. Proposition \ref{maxisonum}), the order of ${\rm Sp}_6(\mathbb{Z})/\varGamma_0(2)$ is $135$.
A complete list of these 135 representatives of ${\rm Sp}_6(\mathbb{Z})/\varGamma_0(2)$ can be found, for example, in \cite[$\S {\rm V}.21$]{Tsuyumine}.\vspace{-0.5mm}
\begin{Rmk}\label{kernel}
{As shown in Proposition \ref{construct}, the symplectic matrix $M_k$ that induces the sign change only in $\vartheta_k/\vartheta_0$ for each $k \in \{1,\ldots,6\}$ belongs to the subgroup $\varGamma_0(2)$.}
In other words, for arbitrary choice of the square roots in line 2 of Algorithm \ref{construct}, the matrix that induces the corresponding sign changes leaves $\langle P,Q,R \rangle$ invariant.
A similar statement holds for line 6.
\end{Rmk}

The action of each $M \in {\rm Sp}_6(\mathbb{Z})/\varGamma_0(2)$ on a squared theta null-point of $A$ can be computed via the equation
\[
    \theta\bigl[\begin{smallmatrix}M.a\\\hspace{-0.3mm}M.b\end{smallmatrix}\bigr](0,M.\varOmega)^2 = c \cdot e^{2\pi ik(M,a,b)}\theta\bigl[\begin{smallmatrix}a\\b\end{smallmatrix}\bigr](0,\varOmega)^2\smallskip
\]
by using Theorem \ref{trans}, where $M.a,M.b$, and $k(M,a,b)$ are defined in (\ref{index}) and (\ref{scalar}).
By applying Algorithm \ref{compute_isogeny} to each of the squared theta null-points obtained above, we can compute a list of squared theta null-points of all the 135 abelian threefolds that are $(2,2,2)$-isogenous to the given abelian threefold $A$.
Here is the pseudocode (Algorithm \ref{compute_isogenies} below) summarizes this procedure:\vspace{-0.7mm}

\begin{algorithm}[htbp]
\begin{algorithmic}[1]
\caption{{\sf \ ComputeAllRichelotIsogenies}}\label{compute_isogenies}
\renewcommand{\algorithmicrequire}{\textbf{Input:}}
\renewcommand{\algorithmicensure}{\textbf{Output:}}
\Require A squared theta null-point $[\vartheta_i^2]_i \in \mathbb{P}^{63}$ of a p.p.a.v.\hspace{1mm}$A$ of dimension 3
\Ensure A list of squared theta null-points of all the p.p.a.v.'s which are $(2,2,2)$-isogenous to $A$
\State Initialize $\mathcal{B} \leftarrow \emptyset$
\ForAll{$M \in {{\rm Sp}_6(\mathbb{Z})/\varGamma_0(2)}$}
\State Write $M = \begin{psmallmatrix} \alpha & \beta\\[-0.3mm] \gamma & \delta\end{psmallmatrix}$ where $\alpha,\beta,\gamma$ and $\delta$ are all $(3 \times 3)$-matrices
\For{$i \in \{0,\ldots,63\}$}
\State Write $i = {a_3a_2a_1b_3b_2b_1}_{(2)}$ in binary and $a \leftarrow (a_1,a_2,a_3),\,b \leftarrow (b_1,b_2,b_3)$
\State Let $k \leftarrow (a\hspace{0.3mm}^t\hspace{-0.3mm}\delta - b\,^t\hspace{-0.3mm}\gamma) \,\cdot\hspace{0.3mm}^t\hspace{-0.3mm}(b\,^t\hspace{-0.3mm}\alpha - a\hspace{0.3mm}^t\hspace{-0.3mm}\beta + 2(\alpha\hspace{0.3mm}^t\hspace{-0.3mm}\beta)_0) - a\hspace{0.3mm}^tb$
\State Let $(a'_1,a'_2,a'_3) \leftarrow a\hspace{0.3mm}^t\hspace{-0.3mm}\delta - b\,^t\hspace{-0.3mm}\gamma + (\gamma\hspace{0.3mm}^t\hspace{-0.3mm}\delta)_0,\,(b'_1,b'_2,b'_3) \leftarrow b\,^t\hspace{-0.5mm}\alpha - a\hspace{0.3mm}^t\hspace{-0.3mm}\beta + (\alpha\hspace{0.3mm}^t\hspace{-0.3mm}\beta)_0$
\State Let ${\vartheta'}_{\!\!j}^2 \!\leftarrow (\hspace{-0.3mm}\sqrt{-1})^{4k} \cdot \vartheta_i^2$ with $j = {a'_3a'_2a'_1b'_3b'_2b'_1}_{(2)}$\vspace{-0.3mm}
\EndFor
\State Add {\sf ComputeOneRichelotIsogeny}$([{\vartheta'}_{\!\!i}^2]_i)$ to $\mathcal{B}$
\EndFor
\\\Return $\mathcal{B}$
\end{algorithmic}
\end{algorithm}\vspace{-1mm}

\noindent We note that it suffices to perform the loop in lines 4--9 only for $i \in I_{\rm even}$ instead, since $\vartheta_i^2 \!= 0$ holds for all $i \in I_{\rm odd}$.

\subsection{Producing the defining equation of a genus-3 curve}\label{restore}
In Section \ref{g=3}, we introduced some relations between a genus-3 curve $C$ and a level-4 theta null-point of ${\rm Jac}(C)$.
However, these relations are not sufficient for recovering $C$ from a given squared theta null-point.
Specifically, there are the following issues as it stands:\smallskip
\begin{itemize}
\item[---] If $N_{\rm van} = 1$ and $S_{\rm van} \neq \{61\}$, how should we restore a defining equation of the corresponding genus-3 hyperelliptic curve?\smallskip
\item[---] If $N_{\rm van} = 0$, we have to take some square roots of $\vartheta_k^2$ to use Theorem \ref{Weber}. How should we determine them?\smallskip
\end{itemize}
\noindent In this subsection, we explain in detail how to do this.
We begin by defining several matrices to address the former problem as follows:
\begin{Def}\label{hyp_matrix}
For each $i \in I_{\rm even}$, we define $\beta_i,\gamma_i$ as follows:
\begin{itemize}
\item If $i \notin \{27,31,45,47,54,55,59,61,62\}$, then we set
\[
    \beta_i \coloneqq {\rm diag}(b_1,b_2,b_3), \quad \gamma_i \coloneqq {\rm diag}(a_1,a_2,a_3).\smallskip
\]
where $i$ is expressed in binary as $i = {a_3a_2a_1b_3b_2b_1}_{(2)}$.
\item Otherwise, we set
{\small \begin{alignat*}{3}
    \beta_{27} &\coloneqq
        \begin{pmatrix}
            1 & 1 & 0\\
            1 & 1 & 0\\
            0 & 0 & 0
        \end{pmatrix}, &\quad \gamma_{27} &\coloneqq
        \begin{pmatrix}
            1 & -1 & 0\\
            -1 & 1 & 0\\
            0 & 0 & 0
        \end{pmatrix},\\
    \beta_{31} &\coloneqq
        \begin{pmatrix}
            1 & 1 & 1\\
            1 & 1 & 1\\
            1 & 1 & 1
        \end{pmatrix}, &\quad \gamma_{31} &\coloneqq
        \begin{pmatrix}
            1 & -1 & 0\\
            -1 & 1 & 0\\
            0 & 0 & 0
        \end{pmatrix},\\
    \beta_{45} &\coloneqq
        \begin{pmatrix}
            1 & 0 & 1\\
            0 & 0 & 0\\
            1 & 0 & 1
        \end{pmatrix}, &\quad \gamma_{45} &\coloneqq
        \begin{pmatrix}
            1 & 0 & -1\\
            0 & 0 & 0\\
            -1 & 0 & 1
        \end{pmatrix},\\
    \beta_{47} &\coloneqq
        \begin{pmatrix}
            1 & 1 & 1\\
            1 & 1 & 1\\
            1 & 1 & 1
        \end{pmatrix}, &\quad \gamma_{47} &\coloneqq
        \begin{pmatrix}
            1 & 0 & -1\\
            0 & 0 & 0\\
            -1 & 0 & 1
        \end{pmatrix},\\
    \beta_{54} &\coloneqq
        \begin{pmatrix}
            0 & 0 & 0\\
            0 & 1 & 1\\
            0 & 1 & 1
        \end{pmatrix}, &\quad \gamma_{54} &\coloneqq
        \begin{pmatrix}
            0 & 0 & 0\\
            0 & 1 & -1\\
            0 & -1 & 1
        \end{pmatrix},\\
    \beta_{55} &\coloneqq
        \begin{pmatrix}
            1 & 1 & 1\\
            1 & 1 & 1\\
            1 & 1 & 1
        \end{pmatrix}, &\quad \gamma_{55} &\coloneqq
        \begin{pmatrix}
            0 & 0 & 0\\
            0 & 1 & -1\\
            0 & -1 & 1
        \end{pmatrix},\\
    \beta_{59} &\coloneqq
        \begin{pmatrix}
            1 & -1 & 0\\
            -1 & 1 & 0\\
            0 & 0 & 0
        \end{pmatrix}, &\quad \gamma_{59} &\coloneqq
    \begin{pmatrix}
        1 & 1 & 1\\
        1 & 1 & 1\\
        1 & 1 & 1
    \end{pmatrix},\\
    \beta_{61} &\coloneqq
        \begin{pmatrix}
            1 & 0 & -1\\
            0 & 0 & 0\\
            -1 & 0 & 1
        \end{pmatrix}, &\quad \gamma_{61} &\coloneqq
        \begin{pmatrix}
            1 & 1 & 1\\
            1 & 1 & 1\\
            1 & 1 & 1
        \end{pmatrix},\\
    \beta_{62} &\coloneqq
        \begin{pmatrix}
            0 & 0 & 0\\
            0 & 1 & -1\\
            0 & -1 & 1
        \end{pmatrix}, &\quad \gamma_{62} &\coloneqq
        \begin{pmatrix}
            1 & 1 & 1\\
            1 & 1 & 1\\
            1 & 1 & 1
        \end{pmatrix}.
\end{alignat*}}
\end{itemize}
It is straightforward to verify that the matrix $P_i \coloneqq \begin{psmallmatrix} \mathbf{1}_3 & \beta_i\\[-0.3mm] \gamma_i & \mathbf{1}_3\end{psmallmatrix}$ belongs to ${\rm Sp}_6(\mathbb{Z})$.
\end{Def}
\begin{Lem}\label{trans_hyp}
Suppose that $\vartheta_0 = 0$. For each $i \in I_{\rm even}$, we have $\vartheta_i(0,P_i\hspace{0.2mm}.\hspace{0.2mm}\varOmega) = 0$.
\end{Lem}
\begin{proof}
With the same notations as in Definition \ref{hyp_matrix}, we have by (\ref{index}) that
\[
    P_i\hspace{0.2mm}.\hspace{0.2mm}a = \frac{1}{2}(\gamma_i{\hspace{0.3mm}}^t\hspace{-0.3mm}\delta_i)_0 = (a_1,a_2,a_3), \quad 
    P_i\hspace{0.2mm}.\hspace{0.2mm}b = \frac{1}{2}(\alpha_i{\hspace{0.3mm}}^t\hspace{-0.5mm}\beta_i)_0 = (b_1,b_2,b_3).\smallskip
\]
for $a=b=0$.
Then, this lemma follows from Theorem \ref{trans}.
\end{proof}

Let $A$ be a principally polarized abelian threefold isomorphic to the Jacobian of a hyperelliptic curve $C$ of genus 3 and suppose that a squared theta null-point of $A$ is given.
It follows from Proposition \ref{van_num} that {there exists a unique} index $i \in I_{\rm even}$ satisfying $\vartheta_i = 0$.
Let $[{\vartheta'}_{\!\!i}^2]_i$ be the squared theta null-point of $A$ obtained by acting a matrix $P_{61}{P_i}^{-1} \in {\rm Sp}_6(\mathbb{Z})$ on the given $[\vartheta_i^2]_i$, and we have ${\vartheta'}_{\!\!61} = 0$ by Lemma \ref{trans_hyp}.
Hence, we can apply Proposition \ref{hyp_restore} to $[{\vartheta'}_{\!\!i}^2]_i \in \mathbb{P}^{63}$, which enables us to restore {a} defining equation of $C$ as follows: 

\newpage
\begin{algorithm}[htbp]
\begin{algorithmic}[1]
\caption{{\sf \ RestoreHyperellipticCurve}}\label{restore_hyper}
\renewcommand{\algorithmicrequire}{\textbf{Input:}}
\renewcommand{\algorithmicensure}{\textbf{Output:}}
\Require A squared theta null-point $[\vartheta_i^2]_i \in \mathbb{P}^{63}$ with $N_{\rm van} = 1$ of a p.p.a.v.\hspace{0.7mm}$A$
\Ensure A hyperelliptic curve $C$ of genus $3$ such that $A \cong {\rm Jac}(C)$\vspace{0.3mm}
\State Let $i \in I_{\rm even}$ such that $\vartheta_i^2 = 0$ and define $M \coloneqq P_{61}{P_i}^{-1}$
\State Write $M = \begin{psmallmatrix} \alpha & \beta\\[-0.3mm] \gamma & \delta\end{psmallmatrix}$ where $\alpha,\beta,\gamma$ and $\delta$ are all $(3 \times 3)$-matrices\vspace{0.1mm}
\For{$i \in \{0,\ldots,63\}$}
\State Write $i = {a_3a_2a_1b_3b_2b_1}_{(2)}$ in binary and $a \leftarrow (a_1,a_2,a_3),\,b \leftarrow (b_1,b_2,b_3)$
\State Let $k \leftarrow (a\hspace{0.3mm}^t\hspace{-0.3mm}\delta - b\,^t\hspace{-0.3mm}\gamma) \,\cdot\hspace{0.3mm}^t\hspace{-0.3mm}(b\,^t\hspace{-0.3mm}\alpha - a\hspace{0.3mm}^t\hspace{-0.3mm}\beta + 2(\alpha\hspace{0.3mm}^t\hspace{-0.3mm}\beta)_0) - a\hspace{0.3mm}^tb$
\State Let $(a'_1,a'_2,a'_3) \leftarrow a\hspace{0.3mm}^t\hspace{-0.3mm}\delta - b\,^t\hspace{-0.3mm}\gamma + (\gamma\hspace{0.3mm}^t\hspace{-0.3mm}\delta)_0,\,(b'_1,b'_2,b'_3) \leftarrow b\,^t\hspace{-0.5mm}\alpha - a\hspace{0.3mm}^t\hspace{-0.3mm}\beta + (\alpha\hspace{0.3mm}^t\hspace{-0.3mm}\beta)_0$
\State Let ${\vartheta'}_{\!\!j}^2 \!\leftarrow (\hspace{-0.3mm}\sqrt{-1})^{4k} \cdot \vartheta_i^2$ with $j = {a'_3a'_2a'_1b'_3b'_2b'_1}_{(2)}$\vspace{-0.3mm}
\EndFor
\State Compute $\lambda_1,\ldots,\lambda_5$ from the squared theta null-point $[{\vartheta'}_{\!\!i}^2]_i \hspace{-0.2mm}\in \mathbb{P}^{63}$ using (\ref{Rosenhain})\\
\Return the curve $y^2 = x(x-1)\prod_{k=1}^5(x-\lambda_k)$
\end{algorithmic}
\end{algorithm}\vspace{0.3mm}

Next, we suppose that a squared theta null-point of $A = {\rm Jac}(C)$ is given, where $C$ is a smooth plane quartic.
In the following, we explain how to recover $C$ according to the method in \cite[$\S 6.4$]{Milio}. It follows from Theorem \ref{Riemann} that\smallskip
\begin{equation}\label{relation1}
    \left\{
    \begin{array}{l}
        \vartheta_0\vartheta_{16}\vartheta_{45}\vartheta_{61} = \vartheta_5\vartheta_{21}\vartheta_{40}\vartheta_{56} - \vartheta_{12}\vartheta_{28}\vartheta_{33}\vartheta_{49},\\
        \vartheta_{12}\vartheta_{21}\vartheta_{33}\vartheta_{56} = \vartheta_5\vartheta_{28}\vartheta_{40}\vartheta_{49} - \vartheta_2\vartheta_{27}\vartheta_{47}\vartheta_{54},\\
        \vartheta_3\vartheta_{20}\vartheta_{32}\vartheta_{55} = \vartheta_2\vartheta_{21}\vartheta_{33}\vartheta_{54} + \vartheta_{12}\vartheta_{27}\vartheta_{47}\vartheta_{56},
    \end{array}
    \right.
\end{equation}\vspace{-2mm}
\begin{equation}
    \left\{
    \begin{array}{l}\label{relation2}
    \vartheta_7\vartheta_{14}\vartheta_{35}\vartheta_{42} = \vartheta_5\vartheta_{12}\vartheta_{33}\vartheta_{40} -\vartheta_{21}\vartheta_{28}\vartheta_{49}\vartheta_{56},\\
    \vartheta_{14}\vartheta_{16}\vartheta_{35}\vartheta_{61} = \vartheta_2\vartheta_{28}\vartheta_{47}\vartheta_{49} - \vartheta_5\vartheta_{27}\vartheta_{40}\vartheta_{54},\\
    \vartheta_7\vartheta_{16}\vartheta_{42}\vartheta_{61} = \vartheta_2\vartheta_{21}\vartheta_{47}\vartheta_{56} -  \vartheta_{12}\vartheta_{27}\vartheta_{33}\vartheta_{54} 
    \end{array}
    \right.\smallskip
\end{equation}
for appropriate inputs $m_1,m_2,m_3,m_4$.
Squaring both sides of (\ref{relation1}), we obtain\smallskip
\begin{align}\label{square}
    \begin{split}
    2\vartheta_5\vartheta_{12}\vartheta_{21}\vartheta_{28}\vartheta_{33}\vartheta_{40}\vartheta_{49}\vartheta_{56} &= \vartheta_5^2\vartheta_{21}^2\vartheta_{40}^2\vartheta_{56}^2 + \vartheta_{12}^2\vartheta_{28}^2\vartheta_{33}^2\vartheta_{49}^2 -\vartheta_0^2\vartheta_{16}^2\vartheta_{45}^2\vartheta_{61}^2,\\[-0.5mm]
    2\vartheta_2\vartheta_5\vartheta_{27}\vartheta_{28}\vartheta_{40}\vartheta_{47}\vartheta_{49}\vartheta_{54} &= \vartheta_5^2\vartheta_{28}^2\vartheta_{40}^2\vartheta_{49}^2 + \vartheta_2^2\vartheta_{27}^2\vartheta_{47}^2\vartheta_{54}^2 - \vartheta_{12}^2\vartheta_{21}^2\vartheta_{33}^2\vartheta_{56}^2,\\[-0.5mm]
    2\vartheta_2\vartheta_{12}\vartheta_{21}\vartheta_{27}\vartheta_{33}\vartheta_{47}\vartheta_{54}\vartheta_{56} &= \vartheta_3^2\vartheta_{20}^2\vartheta_{32}^2\vartheta_{55}^2 - \vartheta_2^2\vartheta_{21}^2\vartheta_{33}^2\vartheta_{54}^2 - \vartheta_{12}^2\vartheta_{27}^2\vartheta_{47}^2\vartheta_{56}^2.
    \end{split}\\[-3.6mm]\nonumber
\end{align}
Then, we can restore a defining equation of a plane quartic $C$, as follows:\vspace{-0.3mm}

\begin{algorithm}[H]
\begin{algorithmic}[1]
\caption{{\sf \ RestorePlaneQuartic}}\label{restore_quartic}
\renewcommand{\algorithmicrequire}{\textbf{Input:}}
\renewcommand{\algorithmicensure}{\textbf{Output:}}
\Require A squared theta null-point $[\vartheta_i^2]_i \in \mathbb{P}^{63}$ with $N_{\rm van} = 0$ of a p.p.a.v.\,$A$
\Ensure A plane quartic $C$ of genus $3$ such that $A \cong {\rm Jac}(C)$
\State Compute $a_{11}^2,a_{{21}}^2,a_{{31}}^2$ by equation (\ref{aij}) and choose their square roots
\State Compute $\vartheta_5\vartheta_{12}\vartheta_{33}\vartheta_{40},\,\vartheta_5\vartheta_{27}\vartheta_{40}\vartheta_{54},\,\vartheta_{12}\vartheta_{27}\vartheta_{33}\vartheta_{54}$ from the values $a_{11},a_{21},a_{31}$
\State Compute $\vartheta_{21}\vartheta_{28}\vartheta_{49}\vartheta_{56},\,\vartheta_2\vartheta_{28}\vartheta_{47}\vartheta_{49},\,\vartheta_2\vartheta_{12}\vartheta_{33}\vartheta_{56}$ from equation (\ref{square})
\State Compute $\vartheta_7\vartheta_{14}\vartheta_{35}\vartheta_{42},\,\vartheta_{14}\vartheta_{16}\vartheta_{35}\vartheta_{61},\,\vartheta_7\vartheta_{16}\vartheta_{42}\vartheta_{61}$ from equation (\ref{relation2})
\State Compute $a_{12},a_{13},a_{22},a_{23},a_{32},$ and $a_{33}$ from equation (\ref{aij})
\State Compute $k_1,k_2,k_3$ from equation (\ref{matrix})
\State Compute the homogeneous linear polynomials $\xi_1,\xi_2,\xi_3$ from Theorem \ref{Weber}\\
\Return the curve $(\xi_1x + \xi_2y - \xi_3z)^2 - 4\xi_1\xi_2xy = 0$
\end{algorithmic}
\end{algorithm}\vspace{-0.9mm}

\noindent We note that the square roots of every $a_{i1}^2$ in line 1 can be chosen arbitrarily, since making the sign change $a_{i1} \mapsto -a_{i1}$ simply induces $a_{i2} \mapsto -a_{i2}$ and $a_{i3} \mapsto -a_{i3}$ for each $i \in \{1,2,3\}$ (cf. \cite[p.\,1361]{Milio}).

\subsection{Validity in characteristic different from 2}\label{algebraic}
All the discussions so far have been carried out over $\mathbb{C}$, but in this subsection, let us discuss their validity over an algebraically closed field of characteristic $p \neq 2$.
For a principally polarized abelian variety $A$ of dimension $g$, we consider the level-2 (algebraic) theta structure
\[
    \theta^A_{\rm lv2}: A \rightarrow \mathbb{P}^{2^g-1}\,;\, P \mapsto \bigl[\theta_{(i_g \cdots i_1)}(P)\bigr]_{i_1,\ldots,i_g \in \{0,1\}}
\]
and the level-4 (algebraic) theta structure
\[
    \theta^A_{\rm lv4}: A \rightarrow \mathbb{P}^{2^{2g}-1}\,;\, P \mapsto \biggl[\theta_{\bigl(\substack{i_g \cdots i_1\\[-0.5mm]j_g \cdots j_1}\bigr)}(P)\biggr]_{i_1,\ldots,i_g,j_1,\ldots,j_g \in \{0,1\}}
\]
of $A$.
In particular, we call $\theta^A_{\rm lv4}(0_A) \in \mathbb{P}^{2^{2g}-1}$ a \emph{level-4 (algebraic) theta null-point} of $A$, where $0_A$ is the identity of $A$.
As stated in \cite[Chapter 4, Proposition 5.2]{Robert}, the level-2 theta structure and the level-4 theta structure are related by the formula
\begin{equation}\label{dupl-algebraic}
    \theta_{\bigl(\substack{i\\[0.1mm]j}\bigr)}(P)^2 = \sum_{t=t_g \cdots t_1} (-1)^{\langle i,t\rangle}\theta_t(0_A)\theta_{j+t}(2P)\ \,\text{for all } P \in A\smallskip
\end{equation}
where $\langle i,t \rangle \coloneqq \sum_{k=1}^g i_kt_k$.\vspace{-0.5mm}
\begin{Rmk}\label{algebraic-analytic}
If $A$ is a complex abelian variety with the period matrix $\varOmega \in \mathcal{H}_g$, then we have
\begin{alignat*}{3}
    \theta_t(P) &= \theta\bigl[\begin{smallmatrix}0\\t/2\end{smallmatrix}\bigr](z_P,\varOmega/2) &&\text{ for all }\, t=t_g \cdots t_1 \in \{0,1\}^g,\\
    \theta_{\bigl(\substack{i\\[0.1mm]j}\bigr)}(P)^2 &= \theta\bigl[\begin{smallmatrix}i/2\\j/2\end{smallmatrix}\bigr](2z_P,\varOmega) &&\text{ for all }\, i=i_g \cdots i_1,\,j = j_g \cdots j_1 \in \{0,1\}^g,
\end{alignat*}
where $z_P \in \mathbb{C}^g$ denotes a preimage of $P \in A$ under the natural surjection $\mathbb{C}^g \to A$ (cf. \cite[Chapter 4, Remark 5.3]{Robert}).
In this case, equation (\ref{dupl-algebraic}) in terms of analytic theta functions coincides with the duplication formula (Theorem \ref{duplication}).
\end{Rmk}
\noindent The correspondence given in Remark \ref{algebraic-analytic} allows us to regard algebraic theta functions as a generalization of the analytic theta functions defined at the beginning of Section \ref{theta}.
In particular, in the case where $A$ is an abelian threefold (i.e., $g=3$), the projective value
\[
    \theta^A_{\rm lv4}(0_A)^2 \coloneqq \biggl[\theta_{\bigl(\substack{i_g \cdots i_1\\[-0.5mm]j_g \cdots j_1}\bigr)}(0_A)^2\biggr]_{i_1,\ldots,i_g,j_1,\ldots,j_g \in \{0,1\}}\smallskip
\]
corresponds to the squared theta null-point of $A$ defined in Section \ref{g=3}.
Therefore, we refer to it as a \emph{squared theta null-point} of $A$ (even if $A$ is not defined over $\mathbb{C}$).\bigskip

Under the above correspondence, the contents of Sections \ref{one}--\ref{restore} are still valid over an algebraically closed field of characteristic $\neq 2$ thanks to the algebraic theory by Mumford \cite{Mumford}.
Specifically,\smallskip
\begin{itemize}
\item[---] We refer to \cite[$\S 2$]{Robert-mail} for the descriptions of the theorems stated in Section \ref{theta} (especially Riemann's theta formula and the theta transformation formula) in terms of algebraic theta functions.\medskip
\item[---] The fact that each formula in Section \ref{g=3} remains valid over an algebraically closed field of characteristic $\neq 2$ is found in \cite[p.\,31]{KNRR} or \cite[Remark 6.3]{Pieper}.\smallskip
\end{itemize}
Since the contents of Sections \ref{one}--\ref{restore} are consequences of these statements, Algorithms \ref{compute_isogeny}, \ref{compute_isogenies}, \ref{restore_hyper}, and \ref{restore_quartic} also work for a squared theta null-point of an abelian threefold over an algebraically closed field of characteristic $\neq 2$.

\subsection{Explicit algorithm for listing superspecial genus-3 curves}\label{complete}
In this subsection, let $p > 7$ be a prime integer and we give an algorithm for listing superspecial genus-3 curves.
First, we need to prepare the list $\Lambda_{2,p}$ of all isomorphism classes of superspecial genus-2 curves in addition to the list $\Lambda_{1,p}$ of all isomorphism classes of supersingular elliptic curves.
Given these inputs, one can list superspecial genus-3 curves as described in Algorithm \ref{main}.\vspace{-0.2mm}

\begin{algorithm}[htbp]
\begin{algorithmic}[1]
\caption{{\sf \ {List}SspGenus3Curves}}\label{main}
\renewcommand{\algorithmicrequire}{\textbf{Input:}}
\renewcommand{\algorithmicensure}{\textbf{Output:}}
\algnewcommand{\IIf}[1]{\State\algorithmicif\ #1\ \algorithmicthen}
\algnewcommand{\EndIIf}{\unskip\ \algorithmicend\ \algorithmicif}
\Require A prime integer $p > 7$ and the two lists $\Lambda_{1,p}$ and $\Lambda_{2,p}$
\Ensure A list $\Lambda_{3,p}$ of isomorphism classes of superspecial genus-3 curves
\State Initialize $\mathcal{L}_1 \leftarrow \emptyset,\,\mathcal{L}_2 \leftarrow \emptyset$, and $\mathcal{S} \leftarrow \emptyset$
\State Let $\mathcal{L}_4$ be a list of isomorphism classes of $E_1 \times E_2 \times E_3$ with $E_1,E_2,E_3 \in \Lambda_{1,p}$
\State Let $\mathcal{L}_3$ be a list of isomorphism classes of $E \times {\rm Jac}(C)$ with $E \in \Lambda_{1,p},C \in \Lambda_{2,p}$
\State Let $\mathcal{S}$ be a list of squared theta null-points of p.p.a.v.\hspace{1mm}in $\mathcal{L}_3 \sqcup \mathcal{L}_4$, and $k \leftarrow 1$
\Repeat
\State Let $[\vartheta^2_i]_i$ be the $k$-th squared theta null-point of $\mathcal{S}$
\ForAll{$[{\vartheta'}_{\!\!i}^2]_i \in {\sf ComputeAllRichelotIsogenies}([\vartheta^2_i]_i)$}
\State Compute the number $N_{\rm van}$ of indices $i \in I_{\rm even}$ such that ${\vartheta'}_{\!\!i}^2 = 0$
\If{$N_{\rm van} = 0$}
    \State Let $inv$ be the Dixmier-Ohno {invariants} of ${\sf RestorePlaneQuartic}([{\vartheta'}_{\!\!i}^2]_i)$
    \If{$inv \notin \mathcal{L}_1$}
    \State Add $inv$ to $\mathcal{L}_1$ and $[{\vartheta'}_{\!\!i}^2]_i$ to the end of $\mathcal{S}$
    \EndIf
\ElsIf{$N_{\rm van} = 1$}
    \State Let $inv$ be the Shioda {invariants} of ${\sf RestoreHyperellipticCurve}([{\vartheta'}_{\!\!i}^2]_i)$
    \If{$inv \notin \mathcal{L}_2$}
    \State Add $inv$ to $\mathcal{L}_2$ and $[{\vartheta'}_{\!\!i}^2]_i$ to the end of $\mathcal{S}$
    \EndIf
\EndIf
\EndFor
\State Set $k \leftarrow k+1$
\Until{$k > \hspace{-0.2mm}\#\mathcal{S}$}\\
\Return the list of genus-3 curves whose invariants belong to $\mathcal{L}_1 \sqcup \mathcal{L}_2$
\end{algorithmic}
\end{algorithm}\vspace{-0.4mm}

\noindent Recall from Section \ref{g=3} that an abelian threefold is classified into Type 1,\,2,\,3, or 4, and $\mathcal{V}_i$ in Algorithm \ref{main} represents the list of all the superspecial abelian threefolds of Type $i$.
In line 4, we have to compute squared theta null-points of all the abelian threefolds of Types 3 and 4.
In the following, we explain how to obtain them briefly:
{\begin{Lem}\label{ell-correspondence}
Let $E: y^2 = x(x-1)(x-\lambda)$ be an elliptic curve defined over a field of characteristic $p \neq 2$.
Then,
\[
    \theta^E_{\rm lv4}(0_E)^2 = \raisebox{0.4mm}{$\bigl[$}\sqrt{\lambda}:\sqrt{\lambda-1}:1:0\raisebox{0.4mm}{$\bigr]$} \in \mathbb{P}^3\smallskip
\]
is a squared theta null-point of $E$.
\end{Lem}
\begin{proof}
It follows from \cite[Chapter 7, $\S {\rm A}.1$]{Robert} that
\[
    \theta^E_{\rm lv2}(0_E) = [\theta_0(0_E):\theta_1(0_E)] \quad \text{with}\ \,\theta_0(0_E)^2 = \sqrt{\lambda}+1,\ \,\theta_1(0_E)^2 = \sqrt{\lambda}-1.\smallskip
\]
This lemma directly follows from equality (\ref{dupl-algebraic}).
\end{proof}}

A squared theta null-point of an elliptic curve is obtained by Lemma \ref{ell-correspondence} above.
Also, applying formulae in \cite[$\S 7.5$]{Gaudry} or \cite[$\S 5.1$]{CR}, one can compute a squared theta null-point of the Jacobian of a genus-2 curve.
Since Lemma \ref{diagonal} allows us to compute a squared theta null-point of their product (i.e., abelian threefolds of Type 3 or 4), one can do computations in line 4.\vspace{-0.5mm}
\begin{Rmk}
To generate the list $\Lambda_{2,p}$, one can alternatively use theta functions, instead of Algorithm \ref{ListUpSSpGenus2Curves} (see \cite[Algorithm 27]{OOKYN} for the specific algorithm).
Using such a method, we get squared theta null-points of all superspecial abelian surfaces in process, and do not need to use formulae in \cite[$\S 7.5$]{Gaudry} or \cite[$\S 5.1$]{CR}.
\end{Rmk}\smallskip
    
In lines 7--20, we compute all the $135$ edges starting from a vertex in turn.
If the codomain corresponds to an abelian threefold of Type 1, we compute the Dixmier-Ohno invariants (cf. \cite{Dixmier},\,\cite{Ohno}) of the underlying smooth plane quartic.
Also, if the codomain corresponds to a Type 2 one, we compute the Shioda invariants (cf. \cite{Shioda}) of the underlying hyperelliptic curve of genus 3.
Since these invariants are given by polynomials of bounded degree, they can be computed by using a constant number of {arithmetic operations over the field over which the curve is defined}.
We remark that there is no need to do anything for the codomain of Types 3 and 4 since such a vertex has already been obtained in lines 2--3.
In this way, we continue to search until we find no new vertices in the graph $\mathcal{G}_3(2,p)$.\vspace{-0.5mm}
\begin{Rmk}
If smooth plane quartics $C$ and $C'\hspace{-0.2mm}$ over an algebraically closed field have different Dixmier-Ohno invariants, then they are not isomorphic to each other.
It has been proven that the converse holds when the characteristic is equal to $0$ or greater than $7$ (we refer to \cite{LLGR} for details).
This is the main reason why we assume that $p > 7$ in Algorithm \ref{main}.
\end{Rmk}
\begin{Rmk}
The condition in line 22 can be replaced by $\#\mathcal{L}_1 + \#\mathcal{L}_2 = \#\Lambda_{3,p}$, where $\#\Lambda_{3,p}$ denotes the number of superspecial genus-3 curves in characteristic $p$, which can be explicitly computed from the formula in \cite[Theorem 3.10]{Brock}.
Although the algorithm terminates faster under the alternative condition, we used the original condition in the following computational experiments, for the reasons explained in Remark \ref{later} below.
\end{Rmk}\smallskip

Summing up the above discussions, we see that Algorithm \ref{main} correctly outputs the list of isomorphism classes of superspecial genus-3 curves in characteristic $p > 7$.
The following theorem was conjectured in an earlier version of this paper, however the recent paper \cite{hash} kindly provided its proof:
\begin{Thm}\label{rational}
Let $[\vartheta_i^2] \in \mathbb{P}^{63}$ be any theta null-point computed in Algorithm \ref{main}. Then, we can take each $\vartheta_i$ such that $\vartheta_i \in \fld$.
\end{Thm}
\begin{proof}
It follows from \cite{AT} that supersingular elliptic curves $y^2 = x(x-1)(x-\lambda)$ are maximal (resp. minimal) when $4 \mid p+1$ (resp. $4 \mid p-1$).
Since all the superspecial principally polarized abelian threefolds computed in Algorithm \ref{main} are connected with these products by a composition of the $(2,2,2)$-isogenies defined over $\fld$, they are also maximal (resp. minimal) when $4 \mid p+1$ (resp. $4 \mid p-1$).
Therefore, thanks to \cite[Example 30]{hash}, their level-4 theta null-points $[\vartheta_i] \in \mathbb{P}^{63}$ are $\mathbb{F}_{p^2}$-rational.
This completes the proof of this theorem.
\end{proof}\smallskip

Finally let us show Theorem \ref{Main1}, which asserts the complexity of Algorithm \ref{main} (even if the condition in line 22 is replaced, the complexity remains the same).
\begin{proof}[(Proof of Theorem \ref{Main1})] 
It follows from Theorem \ref{rational} that all the $\vartheta_i$ appearing in Algorithm \ref{main} belong to $\fld$.
Now, it follows from
\[
    \#\Lambda_{1,p} = O(p), \quad \#\Lambda_{2,p} = O(p^3) \quad {\rm and} \quad \#\Lambda_{3,p} = O(p^6)\smallskip
\]
that the complexity of line 2 and line 3 are given by $\widetilde{O}(p^3)$ and $\widetilde{O}(p^4)$, respectively.
Next, lines 6--21 are repeated $O(p^6)$ times, since ultimately $\#\mathcal{S}$ equals the number of isomorphism classes of superspecial principally polarized abelian threefolds.
{In each iteration, line 7 can be computed with a constant number of arithmetic operations and square root extractions over $\fld$.
Moreover, line 8--19 are repeated exactly 135 times and require only a constant number of arithmetic operations and square root extractions over $\fld$, again.
Therefore, Algorithm \ref{main} can be executed using $O(p^6)$ arithmetic operations and square root extractions over $\fld$.
This amounts to $\widetilde{O}(p^6)$ arithmetic operations over $\fld$, and thus we obtain Theorem \ref{Main1} as desired.}
\end{proof}

\section{Computational results}\label{computation}
In this section, we present some computational results obtained by executing our main algorithm.
We implemented the algorithm with Magma \hspace{-0.5mm}V2.26-10~\cite{MAGMA} and ran it on a machine equipped with an AMD EPYC $7742$ CPU and 2TB of RAM.
Our implementation for listing superspecial genus-3 curves is available at the following GitHub repository:\medskip
\begin{center}
    {\url{https://github.com/Ryo-Ohashi/sspg3list}.}\medskip\vspace{1mm}
\end{center}
In Section \ref{listup}, we provide the results of counting superspecial hyperelliptic genus-3 curves in characteristics $11 \leq p < 100$, and in Section \ref{existence}, we show the existence of such curves for characteristics $7 \leq p < 10000$.

\subsection{Enumeration result of superspecial hyperelliptic genus-3 curves}\label{listup}
The purpose of this subsection is to determine the exact number of isomorphism classes of superspecial hyperelliptic genus-3 curves in characteristics $11 \leq p < 100$.
\begin{Rmk}
By Ekedahl's bound~\cite[Theorem II.1.1]{Ekedahl}, there are no superspecial hyperelliptic genus-3 curves in characteristics $2$, $3$, and $5$.
In characteristic $7$, there is exactly one such curve (cf.\ \cite[Theorem 3.15]{Brock}).
\end{Rmk}

For our purpose, we only need to modify Algorithm \ref{main} to output the cardinalities of lists $\mathcal{L}_1,\mathcal{L}_2,\mathcal{L}_3$, and $\mathcal{L}_4$.
Note that $\#\mathcal{L}_1$ (resp. $\#\mathcal{L}_2$) is equal to the number of isomorphism classes of superspecial non-hyperelliptic (resp. hyperelliptic) curves of genus 3.
Moreover, the cardinalities of $\mathcal{L}_3$ and $\mathcal{L}_4$ are given by\smallskip
\[
    \#\mathcal{L}_3 = \#\Lambda_{1,p} \cdot \#\Lambda_{2,p}, \quad \#\mathcal{L}_4 = \binom{\#\Lambda_{1,p}+2}{3}.\medskip
\]
Here, the numbers $\#\Lambda_{1,p}$ and $\#\Lambda_{2,p}$ can be obtained by \cite{Deuring} and \cite{IKO}, respectively.
The following table (Table \ref{enumerate}) shows the cardinalities of four lists $\mathcal{L}_1,\mathcal{L}_2,\mathcal{L}_3,\mathcal{L}_4$ and the time it took along with additional related information:
\begin{itemize}
\item[---] the column {\tt total} shows the number of times lines 6--21 are executed until the condition $k > \#\mathcal{S}$ is satisfied,\smallskip
\item[---] the column {\tt compl} shows the number of times lines 6--21 are executed until the condition $\#\mathcal{L}_1 + \#\mathcal{L}_2 = \#\Lambda_{3,p}$ is satisfied.
\end{itemize}
\noindent We note that ${\tt total}$ is equal to the sum of $\#\mathcal{L}_1,\#\mathcal{L}_2,\#\mathcal{L}_3$, and $\#\mathcal{L}_4$.

\newpage
\begin{table}[htbp]
    \centering
    \begin{tabular}{c||c|c|c|c||c|c||r}
    char. & \hspace{2mm}$\#\mathcal{L}_1$\hspace{2mm} & \hspace{2mm}$\#\mathcal{L}_2$\hspace{2mm} & \hspace{2mm}$\#\mathcal{L}_3$\hspace{2mm} & \hspace{2mm}$\#\mathcal{L}_4$\hspace{2mm} & {\tt total} & {\tt compl} & Time\,(sec)\\\hline
    $p=11$ & 10 & 1 & 4 & 4 & 19 & 7 & 9.660\\\hline
    $p=13$ & 18 & 1 & 3 & 1 & 23 & 5 & 7.790\\\hline
    $p=17$ & 54 & 2 & 10 & 4 & 70 & 24 & 19.320\\\hline
    $p=19$ & 87 & 4 & 14 & 4 & 109 & 48 & 29.860\\\hline
    $p=23$ & 213 & 9 & 30 & 10 & 262 & 70 & 66.070\\\hline
    $p=29$ & 681 & 10 & 54 & 10 & 755 & 249 & 189.090\\\hline
    $p=31$ & 950 & 27 & 60 & 10 & 1047 & 249 & 258.710\\\hline
    $p=37$ & 2448 & 35 & 93 & 10 & 2586 & 591 & 615.930\\\hline
    $p=41$ & 4292 & 54 & 160 & 20 & 4526 & 1754 & 1070.190\\\hline
    $p=43$ & 5567 & 82 & 180 & 20 & 5849 & 2685 & 1385.360\\\hline
    $p=47$ & 9138 & 125 & 285 & 35 & 9583 & 3277 & 2295.400\\\hline
    $p=53$ & 18032 & 153 & 390 & 35 & 18610 & 6663 & 4484.100\\\hline
    $p=59$ & 33204 & 299 & 624 & 56 & 34183 & 11964 & 8719.060\\\hline
    $p=61$ & 40259 & 262 & 565 & 35 & 41121 & 13015 & 10670.860\\\hline
    $p=67$ & 69132 & 451 & 870 & 56 & 70509 & 25780 & 19646.600\\\hline
    $p=71$ & 96717 & 647 & 1190 & 84 & 98638 & 30861 & 29135.890\\\hline
    $p=73$ & 113778 & 582 & 1098 & 56 & 115514 & 35823 & 34888.620\\\hline
    $p=79$ & 180273 & 942 & 1596 & 84 & 182895 & 51592 & 59625.370 \\\hline
    $p=83$ & 240755 & 1136 & 2080 & 120 & 244091&  70445 & 83296.450\\\hline
    $p=89$ & 362720 & 1402 & 2528 & 120 & 366770 & 122626 & 134065.830\\\hline
    $p=97$ & 602062 & 2002 & 3200 & 120 & 607384 & 208415 & 230401.380
    \end{tabular}\vspace{3mm}
    \caption{The cardinalities of $\mathcal{L}_1,\mathcal{L}_2,\mathcal{L}_3$, and $\mathcal{L}_4$ for $11 \leq p < 100$. We remark that the cardinality of $\mathcal{L}_1$ (resp. $\mathcal{L}_2$) coincides with the number of all isomorphism classes of superspecial non-hyperelliptic (resp. hyperelliptic) curves of genus 3.}\label{enumerate}
\end{table}

Notably, the final results obtained show that $\#\mathcal{L}_1 + \#\mathcal{L}_2 = \#\Lambda_{3,p}$, and therefore our result is consistent with Brock's result~\cite[Theorem 3.10]{Brock}.
\begin{Rmk}\label{later}
As mentioned earlier, the runtime can be reduced by replacing the condition $k > \#\mathcal{S}$ in line 22 with $\#\mathcal{L}_1 + \#\mathcal{L}_2 = \#\Lambda_{3,p}$.
However, in order to verify the consistency of our results with Brock's one, we executed the algorithm without modifying the condition.
\end{Rmk}

\noindent As shown in Table \ref{enumerate}, the ratio ${\tt compl}/{\tt total}$ is approximately $1/3$ for each $p$.
Hence, if the termination condition is replaced with  $\#\mathcal{L}_1 + \#\mathcal{L}_2 = \#\Lambda_{3,p}$, then the runtime would be approximately one-third of the runtime indicated in Table \ref{enumerate}.
For example, under the improved version, listing superspecial genus-3 curves in characteristic $97$ would require roughly one day.

\subsection{Existence of a superspecial hyperelliptic genus-3 curve}\label{existence}
The purpose of this subsection is to show that there exists at least one superspecial hyperelliptic genus-3 curve in characteristic $p$ with $7 \leq p < 10000$.
\begin{Rmk}
Moriya-Kudo's method~\cite{MK} is more efficient for finding such curves, but their method can only generate curves whose automorphism groups contain the Klein $4$-group.
Even if their approach cannot find such a curve for some $p$, there is a possibility that our approach will find them.
\end{Rmk}

First, sufficient conditions on the characteristics $p$ for such a curve to exist are given by the following theorem.
\begin{Thm}
The following assertions are true:\smallskip
\begin{enumerate}
\item The curve $y^2 = x^7-1$ is superspecial if and only if $p \equiv 6 \pmod{7}$.
\item The curve $y^2 = x^7-x$ is superspecial if and only if $p \equiv 3 \pmod{4}$.
\item The curve $y^2 = x^8-1$ is superspecial if and only if $p \equiv 7 \pmod{8}$.
\end{enumerate}
\end{Thm}
\begin{proof}
See \cite[Theorem 2]{Valentini} or \cite[Theorem 3.15]{Brock}.
\end{proof}

By the above theorem, there exists a superspecial hyperelliptic curve of genus $3$ if $p \equiv  6 \pmod{7}$ or $p \equiv 3 \pmod{4}$. To investigate the other cases, we implemented a probabilistic algorithm (Algorithm \ref{exist}) which outputs such a curve:\vspace{-0.9mm}

\begin{algorithm}[htbp]
\begin{algorithmic}[1]
\caption{{\sf \ GenerateSingleSSpHypGenus3Curve}}\label{exist}
\renewcommand{\algorithmicrequire}{\textbf{Input:}}
\renewcommand{\algorithmicensure}{\textbf{Output:}}
\algnewcommand{\IIf}[1]{\State\algorithmicif\ #1\ \algorithmicthen}
\algnewcommand{\EndIIf}{\unskip\ \algorithmicend\ \algorithmicif}
\Require A prime integer $p \geq 7$
\Ensure A superspecial hyperelliptic curve of genus $3$ in characteristic $p$
\State Generate a supersingular elliptic curve $E$
\State Compute a squared theta null-point $[{\vartheta'}_{\!\!i}^2]_i$ of $E^3$ by using Lemma \ref{diagonal}
\Repeat
\State Set $[\vartheta_i^2]_i \leftarrow [{\vartheta'}_{\!\!i}^2]_i$ and choose a matrix $M \in {\rm Sp}_6(\mathbb{Z})/\varGamma_0(2)$ randomly
\State Write $M = \begin{psmallmatrix} \alpha & \beta\\[-0.3mm] \gamma & \delta\end{psmallmatrix}$, where $\alpha,\beta,\gamma$, and $\delta$ are $(3 \times 3)$-matrices
\For{$i \in \{0,\ldots,63\}$}\vspace{0.5mm}
\State Write $i = {a_3a_2a_1b_3b_2b_1}_{(2)}$ in binary and $a \leftarrow (a_1,a_2,a_3),\,b \leftarrow (b_1,b_2,b_3)$
\State Let $k \leftarrow (a\hspace{0.3mm}^t\hspace{-0.3mm}\delta - b\,^t\hspace{-0.3mm}\gamma) \,\cdot\hspace{0.3mm}^t\hspace{-0.3mm}(b\,^t\hspace{-0.3mm}\alpha - a\hspace{0.3mm}^t\hspace{-0.3mm}\beta + 2(\alpha\hspace{0.3mm}^t\hspace{-0.3mm}\beta)_0) - a\hspace{0.3mm}^tb$
\State Let $(a'_1,a'_2,a'_3) \leftarrow a\hspace{0.3mm}^t\hspace{-0.3mm}\delta - b\,^t\hspace{-0.3mm}\gamma + (\gamma\hspace{0.3mm}^t\hspace{-0.3mm}\delta)_0,\,(b'_1,b'_2,b'_3) \leftarrow b\,^t\hspace{-0.5mm}\alpha - a\hspace{0.3mm}^t\hspace{-0.3mm}\beta + (\alpha\hspace{0.3mm}^t\hspace{-0.3mm}\beta)_0$
\State Let ${\vartheta'}_{\!\!j}^2 \!\leftarrow (\hspace{-0.3mm}\sqrt{-1})^{4k} \cdot \vartheta_i^2$ with $j = {a'_3a'_2a'_1b'_3b'_2b'_1}_{(2)}$
\EndFor
\State Compute the number $N_{\rm van}$ of indices $i \in I_{\rm even}$ such that ${\vartheta'}_{\!\!i}^2 = 0$
\Until{$N_{\rm van} = 1$}\\
\Return ${\sf RestoreHyperellipticCurve}([{\vartheta'}_{\!\!i}^2]_i)$
\end{algorithmic}
\end{algorithm}\vspace{-0.7mm}

\noindent Then, one can obtain a superspecial hyperelliptic genus-3 curve as follows:
\begin{itemize}
\item If $p \equiv 6 \pmod{7}$, then output the curve $y^2 = x^7-1$.\smallskip
\item If $p \equiv 3 \pmod{4}$, then output the curve $y^2 = x^7-x$.\smallskip
\item Otherwise, execute the Algorithm \ref{exist}.\smallskip
\end{itemize}
\noindent We remark that Algorithm \ref{exist} may return different outputs depending on randomness in lines $1$ and $4$.\par\medskip

Executing the above procedure, we have checked that a superspecial hyperelliptic genus-3 curve exists for all characteristic $7 \leq p < 10000$. As Algorithm \ref{exist} performs a random walk on $\mathcal{G}_3(2,p)$, the time to output such a curve varies, but we list them below for reference:\smallskip
\begin{itemize}
\item[---] For every $7 \leq p < 10000$, the above procedure ends within 10,000 seconds.
Moreover for most primes $p$, it took less than 3,600 seconds.\smallskip
\item[---] It took 9,952 seconds for $p = 8893$, which is the maximal time for $p < 10000$.\smallskip
\item[---] It took very short time for some $p$ (e.g., about $30$ seconds for $p=9209$).\smallskip
\end{itemize}
Summarizing the discussions in Sections \ref{listup} and \ref{existence} leads us to Theorem \ref{Main2}.

\subsection{Future works}
It would be interesting to list superspecial curves of genus 3 in characteristic $p > 100$.
However, since enumerating in $p = 97$ took about a day, enumerating even in $p = 199$ can be expected to take more than two months (indeed, according to Theorem 1.1, doubling $p$ is expected to increase the computation time by roughly $2^6 = 64$ times).
Hence, improving our main algorithm (Algorithm \ref{main}) is an important task.
This task requires an efficient search to obtain new nodes on the graph $\mathcal{G}_3(2,p)$, and we need to investigate the structure of $\mathcal{G}_3(2,p)$ in detail.\par\smallskip

On the other hand, we can expect from the results in Table \ref{enumerate} that
\begin{Exp}\label{expe}
The number of isomorphism classes of superspecial hyperelliptic curves of genus $3$ in characteristic $p$ is $\Theta(p^4)$. 
\end{Exp}
\noindent Recall that the order of magnitude of all superspecial genus-3 curves is $\Theta(p^6)$, then the proportion of hyperelliptic ones among them can be estimated to be about $1/p^2$ under this expectation.
A theoretical proof of Expectation \ref{expe} remains also one of our future goals.

\vfill
\noindent\hrulefill\smallskip

\bigskip{\small
\textsc{Ryo Ohashi: Graduate School of Information Science and Technology, The University of Tokyo ---  7-3-1 Hongo, Bunkyo-ku, Tokyo, 113-0033, Japan.}\par
{\it E-mail address}: \url{ryo-ohashi@g.ecc.u-tokyo.ac.jp}.

\bigskip
\textsc{Hiroshi Onuki: Graduate School of Information Science and Technology, The University of Tokyo --- 7-3-1 Hongo, Bunkyo-ku, Tokyo, 113-0033, Japan.}
\par
{\it E-mail address}: \url{hiroshi-onuki@g.ecc.u-tokyo.ac.jp}.

\bigskip
\textsc{Momonari Kudo: Fukuoka Institute of Technology --- 3-30-1 Wajiro-higashi, Higashi-ku, Fukuoka, 811-0295, Japan.}\par
{\it E-mail address}: \url{m-kudo@fit.ac.jp}.

\bigskip
\textsc{Ryo Yoshizumi: Joint Graduate Program of Mathematics for Innovation, Kyushu University --- 744 Motooka, Nishi-ku, Fukuoka 819-0395, Japan.}\par
{\it E-mail address}: \url{yoshizumi.ryo.483@s.kyushu-u.ac.jp}.

\bigskip
\textsc{Koji Nuida: Institute of Mathematics for Industry, Kyushu university --- 744 Motooka, Nishi-ku, Fukuoka 819-0395, Japan / National Institute of Advanced Industrial Science and Technology (AIST) --- 2-3-26 Aomi, Koto-ku, Tokyo, 135-0064, Japan}.\par
{\it E-mail address}: \url{nuida@imi.kyushu-u.ac.jp}.}

\begin{thebibliography}{99}
\bibitem{ATY} \textsc{Y.\,Aikawa, R.\,Tanaka and T.\,Yamauchi}: \textit{Isogeny graphs on superspecial abelian varieties: Eigenvalues and connection to Bruhat-Tits buildings}, Canad.\,J.\,Math.\,{\bf 76}, 1891--1916, 2024.
\bibitem{AT} \textsc{R.\,Auer and J.\,Top}: \textit{Legendre elliptic curves over finite fields}, J.\,Number Theory {\bf 95}, 303--312, 2002.
\bibitem{FESTA} \textsc{A.\,Basso, L.\,Maino and G.\,Pope}: \textit{FESTA: Fast encryption from supersingular torsion attacks}, ASIACRYPT 2023, LNCS {\bf 14444}, 98--126, 2023.
\bibitem{MAGMA} \textsc{W.\,Bosma, J.\,Cannon and C.\,Playoust}: \textit{The Magma algebra system. I: The user language}, J.\,Symb.\,Comput.\,{\bf 24}, 235--265, 1997.
\bibitem{Brock} \textsc{B.\,W.\,Brock}: \textit{Superspecial curves of genera two and three}, Ph.D.\,thesis, Princeton University, 1993.
\bibitem{CD} \textsc{W.\,Castryck and T.\,Decru}: \textit{Multiradical isogenies}, Contemp.\,Math.\,{\bf 779}, 57--89, 2022.
\bibitem{CDS} \textsc{W.\,Castryck, T.\,Decru, and B.\,Smith}: \textit{Hash functions from superspecial genus-2 curves using Richelot isogenies}, J.\,Math.\,Cryptol.\,{\bf 14}, 268--292, 2020.
\bibitem{Cosset} \textsc{R.\,Cosset}: \textit{Applications des fonctions thêta à la cryptographie sur courbes hyperelliptiques},\ Ph.D.\,thesis, Université Henri Poincaré, 2011.
\bibitem{CR} \textsc{R.\,Cosset and D.\,Robert}: \textit{Computing $(\ell,\ell)$-isogenies in polynomial time on Jacobians of genus 2 curves}, Math.\,Comp.\,{\bf 84}, 1953--1975, 2015.
\bibitem{SQIsignHD} \textsc{P.\,Dartois, A.\,Leroux, D.\,Robert, and B.\,Wesolowski}: \textit{SQIsignHD: New dimensions in cryptography}, EUROCRYPT 2024, LNCS {\bf 14651}, 3--32, 2024.
\bibitem{DMPR} \textsc{P.\,Dartois, L.\,Maino, G.\,Pope and D.\,Robert}: \textit{An algorithmic approach to $(2,2)$-isogenies in the theta model and applications to isogeny-based cryptography}, ASIACRYPT 2024, LNCS {\bf 15486}, 304--338, 2024.
\bibitem{Deuring} \textsc{M.\,Deuring}: \textit{Die Typen der Multiplikatorenringe elliptischer Funktionenk\"{o}rper}, Abh.\,Math. Semin.\,Hansische Univ.\,{\bf 14}, 197--272, 1941.
\bibitem{Dixmier} \textsc{J.\,Dixmier}: \textit{On the projective invariants of quartic plane curves}, Adv.\,Math.\,{\bf 64}, 279--304, 1987.
\bibitem{Ekedahl} \textsc{T.\,Ekedahl}: \textit{On supersingular curves and abelian varieties}, Math.\,Scand.\,{\bf 60}, 151--178, 1987.
\bibitem{Fiorentino} \textsc{A.\,Fiorentino}: \textit{Weber’s formula for the bitangents of a smooth plane quartic},
Publ.\,Math. Besançon Algèbre Théorie Nr.\,{\bf 2019}/2, 5--17, 2019.
\bibitem{FS1} \textsc{E.\,Florit and B.\,Smith}: \textit{Automorphisms and isogeny graphs of abelian varieties, with applications to the superspecial Richelot isogeny graph}, Contemp.\,Math.\,{\bf 779}, 103--132, 2022.
\bibitem{FS2} \textsc{E.\,Florit and B.\,Smith}: \textit{An atlas of the Richelot isogeny graph}, RIMS Kôkyûroku Bessatsu {\bf B90}, 195--219, 2022.
\bibitem{G2SIDH} \textsc{E.\,V.\,Flynn and Y.\,B.\,Ti}: \textit{Genus two isogeny cryptography}, PQCrypto\,2019, LNCS {\bf 11505}, 286--306, 2019.
\bibitem{Gaudry} \textsc{P.\,Gaudry}: \textit{Fast genus 2 arithmetic based on theta functions}, J.\,Math.\,Cryptol.\,{\bf 1}, 243--265, 2007.
\bibitem{Glass} \textsc{J.\,P.\,Glass}: \textit{Theta constants of genus three}, Compos.\,Math.\,{\bf 40}, 123--137, 1980.
\bibitem{Hashimoto} \textsc{K.\,Hashimoto}: \textit{Class numbers of positive definite ternary quaternion Hermitian forms}, Proc. Japan Acad.\,Ser.\,A Math.\,Sci.\,{\bf 59}, 490--493, 1983.
\bibitem{HLP} \textsc{E.\,W.\,Howe, F.\,Leprévost, and B.\,Poonen}: \textit{Large torsion subgroups of split Jacobians of curves of genus two or three}, Forum Math.\,{\bf 12}, 315--364, 2000.
\bibitem{Ibukiyama} \textsc{T.\,Ibukiyama}: \textit{On rational points of curves of genus 3 over finite fields}, Tôhoku Math.\,J.\,{\bf 45}, 311--329, 1993.
\bibitem{IKO} \textsc{T.\,Ibukiyama, T.\,Katsura and F.\,Oort}: \textit{Supersingular curves of genus two and class numbers}, Compos.\,Math.\,{\bf 57}, 127--152, 1986.
\bibitem{Igusa-inv} \textsc{J.\,Igusa}: \textit{Arithmetic variety of moduli for genus two}, Ann.\,Math.\,{\bf 72}, 612--649, 1960.
\bibitem{Igusa} \textsc{J.\,Igusa}: \textit{Theta functions}, Grundlehren der mathematischen\,Wissenschaften\,{\bf 194}, Springer--Verlag, 1972.
\bibitem{JZ} \textsc{B.\,W.\,Jordan and Y.\,Zaytman}: \textit{Isogeny graphs of superspecial abelian varieties and Brandt matrices}, pre-print, arXiv:\ 2005.09031.
\bibitem{KS} \textsc{S.\,Karati and P.\,Sarkar}: \textit{Kummer for genus one over prime-order fields}, J.\,Cryptology {\bf 33}, 92--129, 2020.
\bibitem{KT} \textsc{T.\,Katsura and K.\,Takashima}: \textit{Counting Richelot isogenies between superspecial abelian surfaces}, Open Book Ser.\,{\bf 4}, 283--300, 2020.
\bibitem{KNT} \textsc{A.\,Kazemifard, A.\,R.\,Naghipour and S.\,Tafazolian}: \textit{A note on superspecial and maximal curves}, Bull.\,Iran.\,Math.\,Soc.\,{\bf 39}, 405--413, 2013.
\bibitem{KNRR} \textsc{M.\,Kirschmer, F.\,Narbonne, C.\,Ritzenthaler and D.\,Robert}: \textit{Spanning the isogeny class of a power of an elliptic curve}, Math.\,Comput.\,{\bf 91}, 401--449, 2022.
\bibitem{Kudo} \textsc{M.\,Kudo}: \textit{Counting isomorphism classes of superspecial curves}, RIMS Kôkyûroku Bessatsu {\bf B90}, 77--95, 2022.
\bibitem{KHH} \textsc{M.\,Kudo, S.\,Harashita, E.\,W.\,Howe}: \textit{Algorithms to enumerate superspecial Howe curves of genus 4}, Open Book Ser.\,{\bf 4}, 301--316, 2020.
\bibitem{hash} \textsc{S.\,Kunzweiler, L.\,Maino, T.\,Moriya, C.\,Petit, G.\,Pope, D.\,Robert, M.\,Stopar, Y.\,B.\,Ti}: \textit{Radical 2-isogenies and cryptographic hash functions in dimensions 1,\,2 and 3}, PKC\,2025, to appear (pre-print: \url{https://ia.cr/2024/1732}).
\bibitem{Lange} \textsc{H.\,Lange and Ch.\,Birkenhake}: \textit{Complex Abelian Varieties}, Grundlehren der mathematischen Wissenschaften {\bf 302}, Springer--Verlag, 2004.
\bibitem{LLGR} \textsc{R.\,Lercier, Q.\,Liu, E.\,L.\,García, C.\,Ritzenthaler}: \textit{Reduction type of smooth plane quartics}, Algebra Number Theory {\bf 15}, 1429--1468, 2021.
\bibitem{LR} \textsc{D.\,Lubicz and D.\,Robert}: \textit{Arithmetic on abelian and Kummer varieties}, Finite Fields Appl. {\bf 39}, 130--158, 2016.
\bibitem{Milio} \textsc{E.\,Milio}: \textit{Computing isogenies between Jacobians of curves of genus 2 and 3}, Math.\,Comp. {\bf 89}, 1331--1364, 2020.
\bibitem{MK} \textsc{T.\,Moriya and M.\,Kudo}: \textit{Some explicit arithmetic on curves of genus three and their applications}, pre-print, arXiv:\ 2209.02926.
\bibitem{Mumford} \textsc{D.\,Mumford}: \textit{On the equations defining abelian varieties. I}, Invent.\,Math.\,{\bf 1}, 287--354, 1966.
\bibitem{Tata2} \textsc{D.\,Mumford}: \textit{Tata Lectures on Theta.\,II}, Progr.\,Math.\,{\bf 43}, Birkh\"{a}user Boston, 1984.
\bibitem{NR} \textsc{E.\,Nart and C.\,Ritzenthaler}: \textit{A new proof of a Thomae-like formula for non hyperelliptic genus 3 curves}, Contemp.\,Math.\,{\bf 686}, 137--155, 2017.
\bibitem{OOKYN} \textsc{R.\,Ohashi, H.\,Onuki, M.\,Kudo, R.\,Yoshizumi and K.\,Nuida}: \textit{Computation of the superspecial $(2,2)$-isogeny graph using theta functions}, RIMS Kôkyûroku {\bf 2280}, 12--22, 2024.
\bibitem{Ohno} \textsc{T.\,Ohno}: \textit{The graded ring of invariants of ternary quartics.\,I}, 2007, unpublished.
\bibitem{Oort} \textsc{F.\,Oort}: \textit{Hyperelliptic supersingular curves}, Prog.\,Math.\,{\bf 89}, 247--284, 1991.
\bibitem{Oyono} \textsc{E.\,Oyono}: \textit{Non-hyperelliptic modular Jacobians of dimension 3}, Math.\,Comp.\,{\bf 78}, 1173--1191, 2009.
\bibitem{Pieper} \textsc{A.\,Pieper}: \textit{Theta nullvalues of supersingular abelian varieties}, J.\,Symb.\,Comput.\,{\bf 123}, Article ID:\,102296, 2024.
\bibitem{Pizer} \textsc{A.\,K.\,Pizer}: \textit{Ramanujan graphs}, AMS/IP Stud.\,Adv.\,Math.\,{\bf 7}, 159--178, 1998.
\bibitem{Robert-mail} \textsc{D.\,Robert}: \textit{Efficient algorithms for abelian varieties and their moduli spaces}, HDR, 2021.
\bibitem{Robert} \textsc{D.\,Robert}: \textit{Some notes on algorithms for abelian varieties}, \url{https://ia.cr/2024/406}.
\bibitem{Shaska} \textsc{T.\,Shaska and G.\,S.\,Wijesiri}: \textit{Theta functions and algebraic curves with automorphisms}, NATO Sci.\,Peace Secur.\,Ser.\,D Inf.\,Commun.\,Secur.\,{\bf 24}, 193--237, 2009.
\bibitem{Shioda} \textsc{T.\,Shioda}: \textit{On the graded ring of invariants of binary octavics}, Am.\,J.\,Math.\,{\bf 89}, 1022--1046, 1967.
\bibitem{Smith} \textsc{B.\,Smith}: \textit{Explicit endomorphisms and correspondences}, Ph.D.\,thesis, University of Sydney, 2005.
\bibitem{Smith2} \textsc{B.\,Smith}: \textit{Isogenies and the discrete logarithm problem in Jacobians of genus 3 hyperelliptic curves}, J.\,Cryptology {\bf 22}, 505--529, 2009.
\bibitem{Takase} \textsc{K.\,Takase}: \textit{A generalization of Rosenhain's normal form for hyperelliptic curves with an application}, Proc.\,Japan Acad.\,Ser.\,A Math.\,Sci.\,{\bf 72}, 162--165, 1996.
\bibitem{Tsuyumine} \textsc{S.\,Tsuyumine}: \textit{On Siegel modular forms of degree three}, Am.\,J.\,Math.\,{\bf 108}, 755--862, 1986.
\bibitem{Valentini} \textsc{R.\,C.\,Valentini}: \textit{Hyperelliptic curves with zero Hasse-Witt matrix}, Manuscr.\,Math.\,{\bf 86}, 185--194, 1995.
\bibitem{Weber} \textsc{H.\,Weber}: \textit{Theorie der Abelschen Functionen vom Geschlecht drei}, Berlin: Druck und Verlag von Georg Reimer, 1876.
\end{thebibliography}
\end{document}